\documentclass[11pt]{article}
\usepackage{xcolor}
\usepackage{geometry}
\usepackage{amssymb}
\usepackage[tbtags]{mathtools}
\usepackage{amsthm, thmtools}
\usepackage{graphicx,epsfig,color}
\usepackage{float}
\usepackage{booktabs, multirow, makecell, lscape}
\usepackage{bm}
\usepackage[ruled, noline, nofillcomment, linesnumbered]{algorithm2e}
\usepackage{hyperref, cleveref}
\usepackage{longtable}
\usepackage[sort&compress]{natbib}

\DontPrintSemicolon
\SetAlgoCaptionSeparator{.}
\SetKwComment{tcc}{$\triangleright$\ }{}

\SetCommentSty{mycommfont}
\SetNlSty{textrm}{}{:}
\SetNlSkip{4pt}

\hypersetup{
	hidelinks,
}

\numberwithin{equation}{section}

\newtheorem{theorem}{Theorem}%[section]
\newtheorem{proposition}[theorem]{Proposition}

\newtheorem{example}[theorem]{Example}
\newtheorem{remark}[theorem]{Remark}
\newcommand{\mbf}[1]{\mathbf{#1}}

\newcommand{\wtilde}[1]{\widetilde{#1}}

\DeclarePairedDelimiter{\norm}{\lVert}{\rVert}
\DeclarePairedDelimiter{\abs}{\lvert}{\rvert}
\DeclarePairedDelimiter{\paren}{\lparen}{\rparen}

\DeclareMathOperator*{\diag}{\mathrm{diag}}

\DeclareMathOperator*{\range}{\mathrm{range}}

\DeclareMathOperator*{\argmin}{\mathrm{argmin}}
\DeclareMathOperator*{\argmax}{\mathrm{argmax}}

\def\mbbr{\mathbb R}

\def\mbf{\mathbf}
\def\rmf{{\rm F}}

\def\nn{\nonumber}

\newcommand{\beq}{\begin{equation}}\newcommand{\eeq}{\end{equation}}
\newcommand{\bit}{\begin{itemize}}\newcommand{\eit}{\end{itemize}}
\newcommand{\bem}{\begin{bmatrix}}\newcommand{\eem}{\end{bmatrix}}

\title{GP-CMRH: An inner product free iterative method for block two-by-two nonsymmetric linear systems}

\author{Kui Du\thanks{School of Mathematical Sciences, Xiamen University, Xiamen 361005, China (kuidu@xmu.edu.cn).},\quad Jia-Jun Fan\thanks{School of Mathematical Sciences, Xiamen University, Xiamen 361005, China ({jiajunfan@stu.xmu.edu.cn}).}}

\date{}

\begin{document}
	\maketitle
	
	\begin{abstract}
		We propose an inner product free iterative method called GP-CMRH for solving block two-by-two nonsymmetric linear systems. 
		GP-CMRH relies on a new simultaneous Hessenberg process that reduces two rectangular matrices to upper Hessenberg form simultaneously, without employing inner products. 
		Compared with GPMR [SIAM J. Matrix Anal. Appl., 44 (2023), pp. 293--311], GP-CMRH requires less computational cost per iteration and may be more suitable for high performance computing and low or mixed precision arithmetic due to its inner product free property.
		Our numerical experiments demonstrate that GP-CMRH and GPMR exhibit comparable convergence behavior (with GP-CMRH requiring slightly more iterations), yet GP-CMRH consumes less computational time in most cases. GP-CMRH significantly outperforms GMRES and CMRH in terms of convergence rate and runtime efficiency. 

\vspace{2mm} 

\vspace{2mm}  
	\end{abstract}
	
	\section{Introduction}
	
	We consider block two-by-two nonsymmetric linear systems of the form
	\begin{equation} \label{eq:gpsys}
		\begin{bmatrix}
			\mathbf{M} & \mathbf{A} \\
			\mathbf{B} & \mathbf{N}
		\end{bmatrix}
		\begin{bmatrix}
			\mathbf{x} \\ \mathbf{y}
		\end{bmatrix} = 
		\begin{bmatrix}
			\mathbf{b} \\ \mathbf{c}
		\end{bmatrix},
	\end{equation}
	where $ \mathbf{M} \in \mathbb{R}^{m \times m} $, $ \mathbf{N} \in \mathbb{R}^{n \times n} $, $ \mathbf{A} \in \mathbb{R}^{m \times n} $, $ \mathbf{B} \in \mathbb{R}^{n \times m} $, $ \mathbf{b} \in \mathbb{R}^m $, and $ \mathbf{c} \in \mathbb{R}^n $. System \eqref{eq:gpsys} is derived from a variety of applications in scientific computing, for example, computational fluid dynamics \cite{elman2014finite}, finite-element discretization of Navier--Stokes equations \cite{elman2002preconditioners}, optimization problems \cite{friedlander2012primal}, and problems obtained by graph partitioning tools such as METIS \cite{karypis1998fast}. 
		
Krylov subspace methods such as BiCG \cite{fletcher1976conjugate}, GMRES \cite{saad1986gmres}, QMR \cite{freund1991qmr}, and Bi-CGSTAB \cite{vorst1992bi} can be employed for solving \eqref{eq:gpsys}. In these methods, the coefficient matrix in \eqref{eq:gpsys} is treated as a whole, and the exploitation of its block structure is generally restricted to the design of preconditioners. 

In recent years, some solvers specifically tailored to the block two-by-two structure of \eqref{eq:gpsys} have been proposed. For symmetric quasi-definite linear systems ($ \mathbf{M} $ is symmetric positive definite, $ \mathbf{N} $ is symmetric negative definite, and $ \mathbf{B} = \mathbf{A}^\top $), Orban and Arioli \cite{orban2017iterative} introduced a family of methods based on the generalized Golub--Kahan bidiagonalization process \cite{arioli2013gener}, extending the methods (e.g., LSQR \cite{paige1982lsqr}, LSMR \cite{fong2011lsmr}, Criag's method \cite{craig1955$n$}) based on the standard Golub--Kahan bidiagonalization process \cite{golub1965calculating}. Based on the Saunders--Simon--Yip tridiagonalization process \cite{saunders1988two}, Montoison and Orban \cite{montoison2021tricg} proposed two methods called TriCG and TriMR for solving symmetric quasi-definite linear systems with nonzero $ \mathbf{b} $ and $ \mathbf{c} $. TriCG and TriMR are equivalent to the preconditioned block CG and block MINRES methods, respectively, but require less storage and costs per iteration. Recently, Du et al. \cite{du2025impro} proposed improved variants of TriCG and TriMR to avoid unlucky terminations. When $ \mathbf{M} $ is symmetric positive definite, $ \mathbf{N} = \mathbf{0} $, and $ \mathbf{B} = \mathbf{A}^\top $, Buttari et al. \cite{buttari2019tridiagonalization} proposed a method called USYMLQR, combining the methods USYMLQ and USYMQR \cite{saunders1988two}. When $ \mathbf{M} $ is invertible and $ \mathbf{N} = \mathbf{0} $, Estrin and Greif \cite{estrin2018spmr} proposed a family of saddle-point solvers called SPMR involving short-term recurrences. We refer to \cite{benzi2005numerical,rozloznik2018saddle} and \cite[Chapter 8]{bai2021matri} for more saddle-point system solvers.	
	
In this paper, we assume that $ \mathbf{M} $ and $ \mathbf{N} $ are invertible, and that the computational costs of $ \mathbf{M}^{-1} \mathbf{v} $ and $ \mathbf{N}^{-1} \mathbf{u} $ are inexpensive. By the block diagonal right-preconditioner $ \mathrm{blkdiag} \paren{\mathbf{M}, \mathbf{N}} $, system \eqref{eq:gpsys} is transformed to
	\[
		{
			\renewcommand{\arraystretch}{0.8}
			\begin{bmatrix}
				\mathbf{I} & \widetilde{\mathbf{A}} \\
				\widetilde{\mathbf{B}} & \mathbf{I}
			\end{bmatrix}
		}
		\begin{bmatrix}
			\widetilde{\mathbf{x}} \\ \widetilde{\mathbf{y}}
		\end{bmatrix} = 
		\begin{bmatrix}
			\mathbf{b} \\ \mathbf{c}
		\end{bmatrix}, \quad
		\widetilde{\mathbf{A}} = \mathbf{A} \mathbf{N}^{-1}, \quad
		\widetilde{\mathbf{B}} = \mathbf{B} \mathbf{M}^{-1}, \quad
		\mathbf{x} = \mathbf{M}^{-1} \widetilde{\mathbf{x}}, \quad
		\mathbf{y} = \mathbf{N}^{-1} \widetilde{\mathbf{y}}.
	\]
	The above system is a special case of the following system
	\begin{equation} \label{eq:sys}
		\begin{bmatrix}
			\lambda \mathbf{I} & \mathbf{A} \\
			\mathbf{B} & \mu \mathbf{I} 
		\end{bmatrix}
		\begin{bmatrix}
			\mathbf{x} \\ \mathbf{y}
		\end{bmatrix} =
		\begin{bmatrix}
			\mathbf{b} \\ \mathbf{c}
		\end{bmatrix},
	\end{equation}
	where $ \lambda \in \mathbb{R} $ and $ \mu \in \mathbb{R} $. When $ \lambda \ne 0 $, the Schur complement system of \eqref{eq:sys} is 
	\[
		(\mathbf{BA} - \lambda \mu \mathbf{I}) \mathbf{y} = \mathbf{Bb} - \lambda \mathbf{c},\quad \mathbf{x} = \lambda^{-1} (\mathbf{b} - \mathbf{Ay}),
	\]
	and when $\mu\neq 0$, the Schur complement system of \eqref{eq:sys} is
	\[
		(\mathbf{AB} - \lambda \mu \mathbf{I}) \mathbf{x} = \mathbf{Ac} - \mu \mathbf{b},\quad \mathbf{y} = \mu^{-1} (\mathbf{c} - \mathbf{Bx}).
	\] Although the above Schur complement systems have smaller size than \eqref{eq:sys}, they may have worse conditioning. We thus focus on iterative methods specifically tailored to the block two-by-two structure of \eqref{eq:sys}.

	Recently, based on the orthogonal Hessenberg reduction, Montoison and Orban \cite{montoison2023gpmr} proposed an iterative method called GPMR for solving \eqref{eq:sys}. GPMR is mathematically equivalent to the block GMRES method applied to 	\begin{equation} \label{eq:blksys}
		\begin{bmatrix}
			\lambda \mathbf{I} & \mathbf{A} \\
			\mathbf{B} & \mu \mathbf{I}
		\end{bmatrix}
		\begin{bmatrix}
			\mathbf{x}^b & \mathbf{x}^c \\
			\mathbf{y}^b & \mathbf{y}^c
		\end{bmatrix} =
		\begin{bmatrix}
			\mathbf{b} & \mathbf{0} \\
			\mathbf{0} & \mathbf{c}
		\end{bmatrix},\quad
		\begin{bmatrix}
			\mathbf{x} \\ \mathbf{y}
		\end{bmatrix} = 
		\begin{bmatrix}
			\mathbf{x}^b \\ \mathbf{y}^b
		\end{bmatrix} + 
		\begin{bmatrix}
			\mathbf{x}^c \\ \mathbf{y}^c
		\end{bmatrix},
	\end{equation} but its storage and computational cost are similar to those of GMRES. The construction of basis vectors in GPMR inherently relies on explicit inner product computations (actually, most commonly used Krylov subspace methods \cite{saad2003iterative} involve inner product computations). For large-scale problems, this reliance on inner products introduces two main key bottlenecks: 	 
	 (1) in parallel computing environments, inner products require global reduction operations and thus degrade parallel efficiency; and (2) in low or mixed precision arithmetic, the accumulation of rounding errors in inner products can lead to a loss of orthogonality of the basis vectors, slower convergence rates, and ultimately, stagnation of iterative methods.
		
	The changing minimal residual Hessenberg method (CMRH) \cite{sadok1999cmrh,heyouni2008new,sadok2012new} is an inner product free method for solving general square linear systems, exhibiting convergence behavior similar to that of GMRES. There exist numerous studies on variants of CMRH in the literature; see, e.g., \cite{heyouni2001globa, duminil2013paral,zhang2014flexi,zhang2014flex,addam2017block,gu2018resta,amini2018block,amini2018weigh,gu2020effic,abdaoui2020simpl,gu2022globa}. Recently, Brown et al. \cite{brown2025h} introduced the H-CMRH method for solving large-scale inverse problem. Shortly after, Brown et al. \cite{brown2024inner} introduced a new Hessenberg process which reduces $ \mathbf{A} $ and its transpose $ \mathbf{A}^\top $ to upper Hessenberg form simultaneously, and proposed an inner product free method called LSLU for solving large-scale inverse problems. The relationship between LSLU and LSQR \cite{paige1982lsqr} parallels that of CMRH and GMRES, in terms of their algorithmic structures and convergence behaviors. Subsequently, Sabat\'e  Landman et al. \cite{sabatelandman2025randomized} introduced the sketched variants of CMRH and LSLU (called sCMRH and sLSLU) to obtain approximate solutions closer to those of GMRES and LSQR than their original counterparts.
	
	To the best of our knowledge, no inner product free iterative methods have been proposed in the literature that are specifically tailored to the block two-by-two structure of \eqref{eq:sys}. To address this gap, we introduce a simultaneous Hessenberg process that reduces two rectangular matrices $\mbf A$ and $\mbf B$ into upper Hessenberg form simultaneously. Leveraging this process, we propose the general partitioned changing minimal residual Hessenberg (GP-CMRH) method, an inner product free iterative solver for solving \eqref{eq:sys}.
		
	This paper is organized as follows. In the reminder of this section, we introduce some notation. In \cref{sec:orth_hessen}, we review the orthogonal Hessenberg reduction used in GPMR. The derivation of the simultaneous Hessenberg process is given in \cref{sec:inner_free_hessen}. In \cref{sec:gpcmrh}, we propose GP-CMRH and give the relations of residual norms between GP-CMRH and GPMR. Numerical experiments and concluding remarks are given in \cref{sec:exp,sec:conclusion}, respectively.
	
	\emph{Notation}. We use uppercase bold letters to denote matrices, and lowercase bold letters to denote column vectors unless otherwise specified. The identity of size $ k \times k $ is denoted by $ \mathbf{I}_k $. We use $ \mathbf{0} $ to denote zero vector or matrix. Let $ \mathbf{e}_i $ denote the $ i $th column of the identity matrix $\mbf I$ whose size is clear from the context. For a vector $ \mathbf{b} $, we use $b_i$, $ \mathbf{b}^\top $, and $ \norm{\mathbf{b}} $ to denote the $i$th element, the transpose, and the Euclidean norm of $ \mathbf{b} $, respectively. For a matrix $ \mathbf{A} $, we use $ \mathbf{A}^\top $, $ \mathbf{A}^\dagger $, $ \norm{\mathbf{A}} $, $ \norm{\mathbf{A}}_{\rmf} $, and $ \range(\mathbf{A}) $ to denote the transpose, the Moore--Penrose inverse, the spectral norm, the Frobenius norm, and the range of $ \mathbf{A} $, respectively. The condition number of a matrix $ \mathbf{A} $ is denoted by $ \kappa(\mathbf{A}):= \norm{\mathbf{A}} \norm{\mathbf{A}^\dagger} $.

	\section{The orthogonal Hessenberg reduction} \label{sec:orth_hessen}
	
	The orthogonal Hessenberg reduction used in GPMR is derived from the following theorem. 
	
	\begin{theorem}[{\cite[Theorem 2.1]{montoison2023gpmr}}] \label{thm:orth_hess}
	For matrices $ \mathbf{A} \in \mathbb{R}^{m \times n} $ and $ \mathbf{B} \in \mathbb{R}^{n \times m} $, there exist orthogonal matrices $ \mathbf{U} \in \mathbb{R}^{n \times n} $ and $ \mathbf{V} \in \mathbb{R}^{m \times m} $ such that
	\begin{equation} \label{eq:hess}
		\mathbf{V}^\top \mathbf{A} \mathbf{U} = \widetilde{\mathbf{H}}, \quad 
		\mathbf{U}^\top \mathbf{B} \mathbf{V} = \widetilde{\mathbf{F}},
	\end{equation}
	where $ \widetilde{\mathbf{H}} = (\wtilde{h}_{i,j}) $ and $ \widetilde{\mathbf{F}} = (\wtilde{f}_{i,j}) $ are both upper Hessenberg.
	\end{theorem}

		 We summarize the orthogonal Hessenberg reduction process in \Cref{alg:gphess}.

		 	\begin{algorithm}[htbp]
%		\setstretch{1.15}
		\caption{Orthogonal Hessenberg reduction}
		\label{alg:gphess}
	
		\KwIn{$ \mathbf{A} \in \mathbb{R}^{m \times n} $, $ \mathbf{B} \in \mathbb{R}^{n \times m} $, and nonzero $ \mathbf{b} \in \mathbb{R}^m $ and $ \mathbf{c} \in \mathbb{R}^n $}
			
		$ \beta \mathbf{v}_1 := \mathbf{b} $, $ \gamma \mathbf{u}_1:= \mathbf{c} $ \tcc*[r]{$ \beta = \lVert \mathbf{b} \rVert $, $ \gamma = \lVert \mathbf{c} \rVert $}
		\For{$ k = 1, 2, \dots $}
		{
			$ \mathbf{v} = \mathbf{A} \mathbf{u}_k $, $ \mathbf{u} = \mathbf{B} \mathbf{v}_k $\;
			\For{$ i = 1, 2, \dots, k $}
			{
				$ \widetilde{h}_{i, k} = \mathbf{v}_i^\top \mathbf{v} $, $ \widetilde{f}_{i, k} = \mathbf{u}_i^\top \mathbf{u} $\;
				$ \mathbf{v} = \mathbf{v} - \widetilde{h}_{i, k} \mathbf{v}_i $\;
				$ \mathbf{u} = \mathbf{u} - \widetilde{f}_{i, k} \mathbf{u}_i $\;
			}
			$ \widetilde{h}_{k+1, k} \mathbf{v}_{k+1} := \mathbf{v} $, $ \widetilde{f}_{k+1, k} \mathbf{u}_{k+1} := \mathbf{u} $ \tcc*[r]{$ \widetilde{h}_{k+1,k} = \lVert \mathbf{v} \rVert $, $ \widetilde{f}_{k+1,k} = \lVert \mathbf{u} \rVert $}
		}
	\end{algorithm}
	
		 Let $ \widetilde{\mathbf{H}}_k $ and $ \widetilde{\mathbf{F}}_k $ be the leading upper-left $ k \times k $ submatrices of $ \widetilde{\mathbf{H}} $ and $ \widetilde{\mathbf{F}} $, respectively. Define 
	\begin{gather*}
		\mathbf{U}_k = 
		\begin{bmatrix}
			\mathbf{u}_1 & \mathbf{u}_2 & \cdots & \mathbf{u}_k
		\end{bmatrix}, \quad 
		\mathbf{V}_k = 
		\begin{bmatrix}
			\mathbf{v}_1 & \mathbf{v_2} & \cdots & \mathbf{v}_k
		\end{bmatrix},
		\shortintertext{and}
		\widetilde{\mathbf{H}}_{k+1, k} = 
		\begin{bmatrix}
			\widetilde{\mathbf{H}}_k \\ \wtilde{h}_{k+1, k} \mathbf{e}_k^\top
		\end{bmatrix}, \quad 
		\widetilde{\mathbf{F}}_{k+1, k} = 
		\begin{bmatrix}
			\widetilde{\mathbf{F}}_k \\ \wtilde{f}_{k+1, k} \mathbf{e}_k^\top
		\end{bmatrix}.
	\end{gather*}
	We have the relations
	\begin{equation} \label{eq:gpre}
		\begin{aligned}
			\mathbf{A} \mathbf{U}_k &= \mathbf{V}_{k+1} \widetilde{\mathbf{H}}_{k+1, k} = \mathbf{V}_k \widetilde{\mathbf{H}}_k + \wtilde{h}_{k+1, k} \mathbf{v}_{k+1} \mathbf{e}_k^\top, \\
			\mathbf{B} \mathbf{V}_k &= \mathbf{U}_{k+1} \widetilde{\mathbf{F}}_{k+1, k} = \mathbf{U}_k \widetilde{\mathbf{F}}_k + \wtilde{f}_{k+1, k} \mathbf{u}_{k+1} \mathbf{e}_k^\top, \\
			\mathbf{V}_k^\top \mathbf{V}_k &= \mathbf{U}_k^\top \mathbf{U}_k = \mathbf{I}_k,\ \widetilde{\mathbf{H}}_k = \mathbf{V}_k^\top \mathbf{A} \mathbf{U}_k,\ \widetilde{\mathbf{F}}_k = \mathbf{U}_k^\top \mathbf{B} \mathbf{V}_k.
		\end{aligned}
	\end{equation}
	From \eqref{eq:gpre} and the orthogonality of $ \mathbf{V}_k $ and $ \mathbf{U}_k $, we have the recursion
	\begin{equation*}
		\begin{aligned}
			\wtilde{h}_{k+1, k} \mathbf{v}_{k+1} = \mathbf{A} \mathbf{u}_k - \sum_{i=1}^{k} \wtilde{h}_{i, k} \mathbf{v}_i, \quad
			\wtilde{f}_{k+1, k} \mathbf{u}_{k+1} = \mathbf{B} \mathbf{v}_k - \sum_{i=1}^{k} \wtilde{f}_{i, k} \mathbf{u}_i,
		\end{aligned}
	\end{equation*}
	where $ \wtilde{h}_{i, k} = \mathbf{v}_i^\top \mathbf{A} \mathbf{u}_k $ and $ \wtilde{f}_{i, k} = \mathbf{u}_i^\top \mathbf{B} \mathbf{v}_k $, for $ i \le k $. 
	
We emphasize that the construction of $ \mathbf{V}_k $ and $ \mathbf{U}_k $ inherently 
	   relies on explicit inner product computations. In the next section, we propose an inner product free process called the simultaneous Hessenberg process in a similar fashion as the Hessenberg process \cite{sadok1999cmrh, heyouni2008new}.
	
	\section{The simultaneous Hessenberg process} \label{sec:inner_free_hessen}
	
	The simultaneous Hessenberg process is derived from the LU factorizations of $ \mathbf{U}_k $ and $ \mathbf{V}_k $. The relationship between the proposed process and \Cref{alg:gphess} is analogous to that between the Hessenberg process and the Arnoldi process \cite{arnoldi1951principle}. We subsequently present the process details. 
	
	Assume that we have the LU factorizations
	\begin{equation} \label{eq:lu}
		\mathbf{V}_k = \mathbf{D}_k \mathbf{R}_{k, V}, \quad
		\mathbf{U}_k = \mathbf{L}_k \mathbf{R}_{k, U}, \quad 
		\mathbf{R}_{k, V},\ \mathbf{R}_{k, U} \in \mathbb{R}^{k \times k},
	\end{equation}
	where $ \mathbf{R}_{k, V} $ and $ \mathbf{R}_{k, U} $ are upper triangular, and $\mathbf{D}_k $ and $ \mathbf{L}_k $ are unit lower trapezoidal (i.e., the first $k$ rows form a unit lower triangular submatrix). We have $ \range(\mathbf{D}_k) = \range(\mathbf{U}_k) $ and $ \range(\mathbf{L}_k) = \range(\mathbf{V}_k) $. Let	
	\[
		\mathbf{D}_k = 
		\begin{bmatrix}
			\mathbf{d}_1 & \mathbf{d}_2 & \cdots & \mathbf{d}_k 
		\end{bmatrix} \in \mathbb{R}^{m \times k}, \quad
		\mathbf{L}_k = 
		\begin{bmatrix}
			\bm{\ell}_1 & \bm{\ell}_2 & \cdots & \bm{\ell}_k
		\end{bmatrix} \in \mathbb{R}^{n \times k}.
	\]	
Combining \eqref{eq:gpre} and \eqref{eq:lu} yields
	\begin{align} 
			\mathbf{A} \mathbf{L}_k &= \mathbf{D}_{k+1} 	\mathbf{R}_{k+1, V} \widetilde{\mathbf{H}}_{k+1, k} \mathbf{R}_{k, U}^{-1} =: \mathbf{D}_{k+1} \mathbf{H}_{k+1, k}, \label{eq:freehess} \\
			\mathbf{B} \mathbf{D}_k &= \mathbf{L}_{k+1} 	\mathbf{R}_{k+1, U} \widetilde{\mathbf{F}}_{k+1, k} \mathbf{R}_{k, V}^{-1} =: \mathbf{L}_{k+1} \mathbf{F}_{k+1, k}, \nn
		\end{align}
	where $ \mathbf{H}_{k+1, k}=(h_{i,j}) $ and $ \mathbf{F}_{k+1, k}=(f_{i,j}) $ are both upper Hessenberg. Next, we establish the recursion for updating $ \mathbf{D}_k $. The derivation of recursion for updating $ \mathbf{L}_k $ is analogous to that of $ \mathbf{D}_k $, and thus we omit it for brevity. Extracting the $ k $th column of \eqref{eq:freehess} yields 
	\begin{equation} \label{eq:dk}
		h_{k+1, k} \mathbf{d}_{k+1} = \mathbf{A} \bm{\ell}_k - \mathbf{D}_k \mathbf{h}_k,\quad
		\mathbf{h}_k := 
		\begin{bmatrix}
			h_{1, k} & h_{2, k} & \cdots & h_{k, k}
		\end{bmatrix}^\top.
	\end{equation}
	Since the first $ k $ elements of $ \mathbf{d}_{k+1} $ are zero, by \eqref{eq:dk} we have
	\[
	\mathbf{h}_k = \left(\begin{bmatrix} \mathbf{I}_k & \mathbf{0} \end{bmatrix} \mathbf{D}_k\right)^{-1} \begin{bmatrix} \mathbf{I}_k & \mathbf{0} \end{bmatrix} \mathbf{A} \bm{\ell}_k =: \mathbf{D}_k^{-} \mathbf{A} \bm{\ell}_k.
	\]
	Then \eqref{eq:dk} can be rewritten as
	\begin{equation} \label{eq:redk}
		h_{k+1, k} \mathbf{d}_{k+1} = (\mathbf{I} - \mathbf{D}_k \mathbf{D}_k^{-}) \mathbf{A} \bm{\ell}_k.
	\end{equation}
	The matrix $ \mathbf{D}_k^{-} = \left(\begin{bmatrix} \mathbf{I}_k & \mathbf{0} \end{bmatrix} \mathbf{D}_k\right)^{-1} \begin{bmatrix} \mathbf{I}_k & \mathbf{0} \end{bmatrix} $ satisfies
	\[
		\mathbf{D}_k^{-} \mathbf{D}_k \mathbf{D}_k^{-} = \mathbf{D}_k^{-}, \quad
		\mathbf{D}_k \mathbf{D}_k^{-} \mathbf{D}_k = \mathbf{D}_k.
	\]
	We now proceed to exploit the recurrence relation inherent in $ \mathbf{D}_k^{-} $. Since
	\begin{equation} \label{eq:Dk}
		\begin{split}
			\mathbf{D}_k^{-} &= \left(\begin{bmatrix} \mathbf{I}_k & \mathbf{0} \end{bmatrix} \mathbf{D}_k\right)^{-1} 
			\begin{bmatrix} \mathbf{I}_k & \mathbf{0} \end{bmatrix} = 
			\begin{bmatrix}
				\begin{bmatrix} \mathbf{I}_{k-1} & \mathbf{0} \end{bmatrix} \mathbf{D}_{k-1} & \mathbf{0} \\
				\mathbf{e}_{k}^\top \mathbf{D}_{k-1} & 1
			\end{bmatrix}^{-1}
			\begin{bmatrix}
				\begin{bmatrix} \mathbf{I}_{k-1} & \mathbf{0} \end{bmatrix} \\ \mathbf{e}_k^\top
			\end{bmatrix} \\
			&= 
			\begin{bmatrix}
				\left( \begin{bmatrix} \mathbf{I}_{k-1} & \mathbf{0} \end{bmatrix} \mathbf{D}_{k-1} \right)^{-1} & \mathbf{0} \\
				-\mathbf{e}_k^\top \mathbf{D}_{k-1} \left( \begin{bmatrix} \mathbf{I}_{k-1} & \mathbf{0} \end{bmatrix} \mathbf{D}_{k-1} \right)^{-1} & 1
			\end{bmatrix}
			\begin{bmatrix}
				\begin{bmatrix} \mathbf{I}_{k-1} & \mathbf{0} \end{bmatrix} \\ \mathbf{e}_k^\top
			\end{bmatrix} \\
			&= 
			\begin{bmatrix}
				\mathbf{D}_{k-1}^{-} \\
				\mathbf{e}_k^\top \left(\mathbf{I} - \mathbf{D}_{k-1} \mathbf{D}_{k-1}^{-}\right)
			\end{bmatrix},
		\end{split}
	\end{equation}
	we have 
	\begin{equation} \label{eq:reDk}
		\begin{split}
			\mathbf{I} - \mathbf{D}_k \mathbf{D}_k^{-} &= \mathbf{I} - 
			\begin{bmatrix} 
				\mathbf{D}_{k-1} & \mathbf{d}_k 
			\end{bmatrix} 
			\begin{bmatrix}
				\mathbf{D}_{k-1}^{-} \\
				\mathbf{e}_k^\top \left(\mathbf{I} - \mathbf{D}_{k-1} \mathbf{D}_{k-1}^{-}\right)
			\end{bmatrix} \\
			&= (\mathbf{I} - \mathbf{d}_k \mathbf{e}_k^\top) (\mathbf{I} - \mathbf{D}_{k-1} \mathbf{D}_{k-1}^{-}) \\
			\shortvdotswithin{=} %&=(\mathbf{I} - \mathbf{d}_k \mathbf{e}_k^\top)(\mathbf{I} - \mathbf{d}_{k-1} \mathbf{e}_{k-1}^\top)\cdots (\mathbf{I} - \mathbf{d}_1 \mathbf{e}_1^\top)\\
			&= \prod_{j=k}^{1} (\mathbf{I} - \mathbf{d}_j \mathbf{e}_j^\top).
		\end{split}
	\end{equation}
	From the above discussion, we obtain the following result.
	
	\begin{proposition} \label{thm:updk} The columns of unit lower trapezoidal matrices $\mbf D_{k+1}$ and $\mbf L_{k+1}$ satisfy the following recursions
		\begin{align}
			h_{k+1, k} \mathbf{d}_{k+1} &= \prod_{j=k}^{1} (\mathbf{I} - \mathbf{d}_j \mathbf{e}_j^\top) \mathbf{A} \bm{\ell}_k, & h_{i, k} &= \mathbf{e}_i^\top \mathbf{h}_k = \mathbf{e}_i^\top \prod_{j=i-1}^{1} (\mathbf{I} - \mathbf{d}_j \mathbf{e}_j^\top) \mathbf{A} \bm{\ell}_k, \label{eq:updk} \shortintertext{and}
			f_{k+1, k} \bm{\ell}_{k+1} &= \prod_{j=k}^{1} (\mathbf{I} - \bm{\ell}_j \mathbf{e}_j^\top) \mathbf{B} \mathbf{d}_k, & f_{i, k} &= \mathbf{e}_i^\top \mathbf{f}_k = \mathbf{e}_i^\top \prod_{j=i-1}^{1} (\mathbf{I} - \bm{\ell}_j \mathbf{e}_j^\top) \mathbf{B} \mathbf{d}_k. \notag % \label{eq:uplk}
		\end{align}
	\end{proposition}
	
	\begin{proof}
		From \eqref{eq:dk} and \eqref{eq:reDk}, we have
		\begin{align*}
			h_{k+1, k} \mathbf{d}_{k+1} = \mathbf{A} \bm{\ell}_k - \mathbf{D}_k \mathbf{h}_k = (\mathbf{I} - \mathbf{D}_k \mathbf{D}_k^{-}) \mathbf{A} \bm{\ell}_k = \prod_{j=k}^{1} (\mathbf{I} - \mathbf{d}_j \mathbf{e}_j^\top) \mathbf{A} \bm{\ell}_k.
		\end{align*}
		From recursion \eqref{eq:Dk}, we have 
		\[
			\mathbf{e}_i^\top \mathbf{D}_k^{-} = \mathbf{e}_i^\top \mathbf{D}_{k-1}^{-} = \dotsb = \mathbf{e}_i^\top \mathbf{D}_i^\top = \mathbf{e}_i^\top (\mathbf{I} - \mathbf{D}_{i-1} \mathbf{D}_{i-1}^{-}).
		\]
		Therefore,
		\[
			h_{i, k} = \mathbf{e}_i^\top\mbf h_k= \mathbf{e}_i^\top \mathbf{D}_k^{-} \mathbf{A} \bm{\ell}_k = \mathbf{e}_i^\top (\mathbf{I} - \mathbf{D}_{i-1} \mathbf{D}_{i-1}^{-}) \mathbf{A} \bm{\ell}_k = \mathbf{e}_i^\top \prod_{j=i-1}^{1} (\mathbf{I} - \mathbf{d}_j \mathbf{e}_j^\top) \mathbf{A} \bm{\ell}_k.
		\]
		The proof of the recursions of $ \bm{\ell}_{k+1} $ and $ f_{i, k} $ is analogous, and thus we omit the details.
	\end{proof}
	
	We summarize the implementation details of the simultaneous Hessenberg process in \Cref{alg:freehess}. \\
	
		\begin{algorithm}[H]
		% \setstretch{1.15}
		\caption{Simultaneous Hessenberg process}
		\label{alg:freehess}
		
		\KwIn{$ \mathbf{A} \in \mathbb{R}^{m \times n} $, $ \mathbf{B} \in \mathbb{R}^{n \times m} $, and nonzero $ \mathbf{b} \in \mathbb{R}^m $ and $ \mathbf{c} \in \mathbb{R}^n $}
			
		$ \beta = \mathbf{e}_1^\top \mathbf{b} $, $ \mathbf{d}_1 = \mathbf{b} / \beta $\;
		$ \gamma = \mathbf{e}_1^\top \mathbf{c} $, $ \bm{\ell}_1 = \mathbf{c} / \gamma $\;
			
		\For{$ k = 1, 2, \dots $}
		{
			$ \mathbf{d} = \mathbf{A} \bm{\ell}_k $, $ \bm{\ell} = \mathbf{B} \mathbf{d}_k $\;
				
			\For{$ i = 1, 2, \dots, k $}
			{
				$ h_{i, k} = \mathbf{d}(i) $, $ f_{i, k} = \bm{\ell}(i) $\;
				$ \mathbf{d} = \mathbf{d} - h_{i, k} \mathbf{d}_i $\;
				$ \bm{\ell} = \bm{\ell} - f_{i, k} \bm{\ell}_i $\;
			}
			$ h_{k+1, k} = \mathbf{d}(k+1) $, $ \mathbf{d}_{k+1} = \mathbf{d} / h_{k+1, k} $\;
			$ f_{k+1, k} = \bm{\ell}(k+1) $, $ \bm{\ell}_{k+1} = \bm{\ell} / f_{k+1, k} $\;
		}
	\end{algorithm}
	
	The MATLAB-type notation ``$ \mathbf{d}(i) $'' in \Cref{alg:freehess} denotes the $ i $th element of $ \mathbf{d} $. When $ h_{k+1, k} = 0 $ or $ f_{k+1, k} = 0 $, \Cref{alg:freehess} terminates. From lines 10 and 11 of \Cref{alg:freehess}, we know that $ h_{k+1, k} $ and $ f_{k+1, k} $ are the $ (k+1) $st elements of $ \mathbf{d} $ and $ \bm{\ell} $, respectively. There exists a potential scenario where $ \mathbf{d}_{k+1} \neq \mathbf{0} $ or $ \bm{\ell}_{k+1} \neq \mathbf{0} $ when the termination occurs. To address this issue, we use the pivoting approach proposed in \cite{heyouni2008new}.

		Let $ \mathbf{D}_k $ and $ \mathbf{L}_k $ be obtained from the pivoted LU factorizations $\mathbf{V}_k = \mathbf{D}_k \mathbf{R}_{k, V}$ and $
		\mathbf{U}_k = \mathbf{L}_k \mathbf{R}_{k, U}$, respectively. With the introduction of pivoting, $ \mathbf{D}_k $ and $ \mathbf{L}_k $ are no longer unit lower trapezoidal, but there exist permutations $ \mathbf{P}_k \in \mathbb{R}^{m \times m} $ and $ \mathbf{Q}_k \in \mathbb{R}^{n \times n} $ so that $ \mathbf{P}_k \mathbf{D}_k $ and $ \mathbf{Q}_k \mathbf{L}_k $ are unit lower trapezoidal. Next, we establish the recursion for updating $ \mathbf{D}_k $. The recursion for $ \mathbf{L}_k $ is analogous, and thus we omit the details for brevity.	Let $ \mathbf{P}_k = \mathbf{G}_k \mathbf{G}_{k-1} \cdots \mathbf{G}_1 \in \mathbb{R}^{m \times m} $ be the product of $ k $ permutations. Define
	\[
		\widetilde{\mathbf{d}}_{k+1} :=  \mathbf{A} \bm{\ell}_k - \mathbf{D}_k \mathbf{h}_k.
	\]
	Let $ \mathbf{G}_{k+1}$ be the permutation which exchanges the $ i_0 $'th and the $ (k+1) $st elements of $ \mathbf{P}_k \widetilde{\mathbf{d}}_{k+1} $, where $$ i_0 \in \argmax_{k+1 \le i \le m} \abs{\mathbf{e}_i^\top \mathbf{P}_k \widetilde{\mathbf{d}}_{k+1}}. $$ Multiplying $ \mathbf{P}_{k+1} = \mathbf{G}_{k+1} \mathbf{P}_k $ on both sides of \eqref{eq:dk} gives
	\begin{equation} \label{eq:pdk}
		h_{k+1, k} \mathbf{P}_{k+1} \mathbf{d}_{k+1} = \mathbf{P}_{k+1} \mathbf{A} \bm{\ell}_k - \mathbf{P}_{k+1} \mathbf{D}_k \mathbf{h}_k.
	\end{equation}
	Since the first $ k $ elements of $ \mathbf{P}_{k+1} \mathbf{d}_{k+1} $ are zero and $ \begin{bmatrix} \mathbf{I}_k & \mathbf{0} \end{bmatrix} \mathbf{G}_{k+1} = \begin{bmatrix} \mathbf{I}_k & \mathbf{0} \end{bmatrix} $, extracting the first $ k $ rows of \eqref{eq:pdk} yields 
	\begin{align*}
		\mathbf{0} 
		&= \begin{bmatrix}
			\mathbf{I}_k & \mathbf{0}
		\end{bmatrix} 
		\mathbf{P}_{k+1} (\mathbf{A} \bm{\ell}_k - \mathbf{D}_k \mathbf{h}_k) 
		= \begin{bmatrix}
			\mathbf{I}_k & \mathbf{0}
		\end{bmatrix} \mathbf{G}_{k+1} \mathbf{P}_k (\mathbf{A} \bm{\ell}_k - \mathbf{D}_k \mathbf{h}_k) \\
		&= \begin{bmatrix}
			\mathbf{I}_k & \mathbf{0}
		\end{bmatrix} \mathbf{P}_k (\mathbf{A} \bm{\ell}_k - \mathbf{D}_k \mathbf{h}_k).
	\end{align*}
	Therefore, $ \mathbf{h}_k $ is given by
	\begin{equation} \label{eq:phk}
		\mathbf{h}_k = \left(
		\begin{bmatrix}
			\mathbf{I}_k & \mathbf{0} 
		\end{bmatrix}
		\mathbf{P}_k \mathbf{D}_k \right)^{-1} 
		\begin{bmatrix}
			\mathbf{I}_k & \mathbf{0}
		\end{bmatrix} \mathbf{P}_k \mathbf{A} \bm{\ell}_k 
		=: (\mathbf{P}_k \mathbf{D}_k)^{-} \mathbf{P}_k \mathbf{A} \bm{\ell}_k.
	\end{equation}
	Substituting \eqref{eq:phk} into \eqref{eq:pdk} yields
	\begin{equation} \label{eq:repdk}
		h_{k+1, k} \mathbf{d}_{k+1} = \mathbf{A}\bm{\ell}_k - \mathbf{D}_k (\mathbf{P}_k \mathbf{D}_k)^{-} \mathbf{P}_k \mathbf{A} \bm{\ell}_k = (\mathbf{I} - \mathbf{D}_k (\mathbf{P}_k \mathbf{D}_k)^{-} \mathbf{P}_k) \mathbf{A} \bm{\ell}_k.
	\end{equation}
	Analogous to the case without pivoting, we have
	\begin{subequations} \label{eq:PDk}
		\begin{equation}
			\begin{split}
				(\mathbf{P}_k \mathbf{D}_k)^{-} \mathbf{P}_k
				&= \left(\begin{bmatrix} \mathbf{I}_k & \mathbf{0} \end{bmatrix} \mathbf{P}_k \mathbf{D}_k\right)^{-1}
				\begin{bmatrix} \mathbf{I}_k & \mathbf{0} \end{bmatrix} \mathbf{P}_k \\
				&= 
				\begin{bmatrix}
					\begin{bmatrix} \mathbf{I}_{k-1} & \mathbf{0} \end{bmatrix} \mathbf{P}_{k-1} \mathbf{D}_{k-1} & \mathbf{0} \\
					\mathbf{e}_{k}^\top \mathbf{P}_k \mathbf{D}_{k-1} & 1
				\end{bmatrix}^{-1}
				\begin{bmatrix}
					\begin{bmatrix} \mathbf{I}_{k-1} & \mathbf{0} \end{bmatrix} \mathbf{P}_{k-1} \\
					\mathbf{e}_k^\top \mathbf{P}_k
				\end{bmatrix} \\
				&= 
				\begin{bmatrix}
					\left( \begin{bmatrix} \mathbf{I}_{k-1} & \mathbf{0} \end{bmatrix} \mathbf{P}_{k-1} \mathbf{D}_{k-1} \right)^{-1} & \mathbf{0} \\
					-\mathbf{e}_k^\top \mathbf{P}_k \mathbf{D}_{k-1} \left( \begin{bmatrix} \mathbf{I}_{k-1} & \mathbf{0} \end{bmatrix} \mathbf{P}_{k-1} \mathbf{D}_{k-1} \right)^{-1} & 1
				\end{bmatrix}
				\begin{bmatrix}
					\begin{bmatrix} \mathbf{I}_{k-1} & \mathbf{0} \end{bmatrix} \mathbf{P}_{k-1} \\ \mathbf{e}_k^\top \mathbf{P}_k
				\end{bmatrix} \\
				&= 
				\begin{bmatrix}
					(\mathbf{P}_{k-1} \mathbf{D}_{k-1})^{-} \mathbf{P}_{k-1} \\
					\mathbf{e}_k^\top \mathbf{P}_k \paren[\big]{\mathbf{I} - \mathbf{D}_{k-1} \paren{\mathbf{P}_{k-1} \mathbf{D}_{k-1}}^{-} \mathbf{P}_{k-1}}
				\end{bmatrix},
			\end{split}
		\end{equation}
		and
		\begin{equation}
			\begin{split}
				\mathbf{I} - \mathbf{D}_k \paren{\mathbf{P}_k \mathbf{D}_k}^{-} \mathbf{P}_k &= \mathbf{I} - 
				\begin{bmatrix} 
					\mathbf{D}_{k-1} & \mathbf{d}_k 
				\end{bmatrix} 
				\begin{bmatrix}
					\paren{\mathbf{P}_{k-1} \mathbf{D}_{k-1}}^{-} \mathbf{P}_{k-1} \\
					\mathbf{e}_k^\top \mathbf{P}_k \paren[\big]{\mathbf{I} - \mathbf{D}_{k-1} \paren{\mathbf{P}_{k-1} \mathbf{D}_{k-1}}^{-} \mathbf{P}_{k-1}}
				\end{bmatrix} \\
				&= \paren[\big]{\mathbf{I} - \mathbf{d}_k \mathbf{e}_k^\top \mathbf{P}_k} \paren[\big]{\mathbf{I} - \mathbf{D}_{k-1} \paren{\mathbf{P}_{k-1} \mathbf{D}_{k-1}}^{-} \mathbf{P}_{k-1}} \\
				\shortvdotswithin{=}
				&= \prod_{j=k}^{1} \paren[\big]{\mathbf{I} -  \mathbf{d}_j \mathbf{e}_j^\top \mathbf{P}_j}.
			\end{split}
		\end{equation}
	\end{subequations}
	From the above discussion, we obtain the following result.
	\begin{proposition} \label{thm:uppdk}
		Let $ \mathbf{D}_k $ and $ \mathbf{L}_k $ be the matrices in the pivoted LU factorizations $\mathbf{V}_k = \mathbf{D}_k \mathbf{R}_{k, V}$ and $
		\mathbf{U}_k = \mathbf{L}_k \mathbf{R}_{k, U}$, respectively. Let $\mbf P_k$ and $\mbf Q_k$ be permutations so that $ \mathbf{P}_k \mathbf{D}_k $ and $ \mathbf{Q}_k \mathbf{L}_k $ are unit lower trapezoidal. The columns of $\mbf D_{k+1}$ and $\mbf L_{k+1}$ satisfy the following recursions
		\begin{align}
			h_{k+1, k} \mathbf{d}_{k+1} &= \prod_{j=k}^{1} (\mathbf{I} - \mathbf{d}_j \mathbf{e}_j^\top \mathbf{P}_j) \mathbf{A} \bm{\ell}_k, &\ 
			h_{i, k} &= \mathbf{e}_i^\top \mathbf{h}_k = \mathbf{e}_i^\top \mathbf{P}_i \prod_{j=i-1}^{1} (\mathbf{I} - \mathbf{d}_j \mathbf{e}_j^\top \mathbf{P}_j) \mathbf{A} \bm{\ell}_k, \label{eq:uppdk} \shortintertext{and}
			f_{k+1, k} \bm{\ell}_{k+1} &= \prod_{j=k}^{1} (\mathbf{I} - \bm{\ell}_j \mathbf{e}_j^\top \mathbf{Q}_j) \mathbf{B} \mathbf{d}_k, &\ f_{i, k} &= \mathbf{e}_i^\top \mathbf{f}_k = \mathbf{e}_i^\top \mathbf{Q}_i \prod_{j=i-1}^{1} (\mathbf{I} - \bm{\ell}_j \mathbf{e}_j^\top \mathbf{Q}_j) \mathbf{B} \mathbf{d}_k. \notag % \label{eq:upplk}
		\end{align}
	\end{proposition}
	
	\begin{proof}
		It follows from \eqref{eq:repdk} and \eqref{eq:PDk} that
		\[
			h_{k+1, k} \mathbf{d}_{k+1} = \bigl(\mathbf{I} - \mathbf{D}_k (\mathbf{P}_k \mathbf{D}_k)^{-} \mathbf{P}_k \bigr) \mathbf{A} \bm{\ell}_k 
			= \prod_{j=k}^{1} \bigl(\mathbf{I} - \mathbf{d}_j \mathbf{e}_j^\top \mathbf{P}_j \bigr) \mathbf{A} \bm{\ell}_k.
		\] 
		Since 
		\[
			\mathbf{e}_i^\top (\mathbf{P}_k \mathbf{D}_k)^{-} \mathbf{P}_k = 
			\mathbf{e}_i^\top (\mathbf{P}_{k-1} \mathbf{D}_{k-1})^{-} \mathbf{P}_{k-1} = \dotsb = 
			\mathbf{e}_i^\top (\mathbf{P}_i \mathbf{D}_i)^{-} \mathbf{P}_i,
		\]
		we have 
		\begin{align*}
			h_{i,k} &= \mathbf{e}_i^\top (\mathbf{P}_k \mathbf{D}_k)^{-} \mathbf{P}_k \mathbf{A} \bm{\ell}_k = \mathbf{e}_i^\top (\mathbf{P}_i \mathbf{D}_i)^{-} \mathbf{P}_i \mathbf{A} \bm{\ell}_k \\
			&= \mathbf{e}_i^\top \mathbf{P}_i \bigl(\mathbf{I} - \mathbf{D}_{i-1} (\mathbf{P}_{i-1} \mathbf{D}_{i-1})^{-} \mathbf{P}_{i-1} \bigr) \mathbf{A} \bm{\ell}_k \\
			&= \mathbf{e}_i^\top \mathbf{P}_i \prod_{j=i-1}^{1} \bigl( \mathbf{I} - \mathbf{d}_j \mathbf{e}_j^\top \mathbf{P}_j \bigr) \mathbf{A} \bm{\ell}_k.
		\end{align*}
		The proof of the recursions of $ \bm{\ell}_{k+1} $ and $ f_{i, k} $ is analogous, and thus we omit the details.
	\end{proof}

	We summarize the implementation details of the simultaneous Hessenberg process with pivoting in \Cref{alg:pfreehess}. \\
	
	\begin{algorithm}[H]
%		\setstretch{1.15}
		\caption{Simultaneous Hessenberg process with pivoting}
		\label{alg:pfreehess}
		
		\KwIn{$ \mathbf{A} \in \mathbb{R}^{m \times n} $, $ \mathbf{B} \in \mathbb{R}^{n \times m} $, and nonzero $ \mathbf{b} \in \mathbb{R}^m $ and $ \mathbf{c} \in \mathbb{R}^n $}
			
		$ \mathbf{p} = \begin{bmatrix} 1 & 2 & \cdots & m \end{bmatrix}^\top $, $ \mathbf{q} = \begin{bmatrix} 1 & 2 & \cdots & n \end{bmatrix}^\top $\;
			
		Determine $ i_0 $ and $ j_0 $ such that $ \abs{b_{i_0}} = \max_{1 \le i \le m} \abs{\mathbf{e}_i^\top \mathbf{b}} $ and $ \abs{c_{j_0}} = \max_{1 \le j \le n} \abs{\mathbf{e}_j^\top \mathbf{c}} $\;
			
		$ \beta = b_{i_0} $, $ \mathbf{d}_1 = \mathbf{b} / \beta $\;
		$ \gamma = c_{j_0} $, $ \bm{\ell}_1 = \mathbf{c} / \gamma $\;
		$ \mathbf{p}(1) \leftrightharpoons \mathbf{p}(i_0) $, $ \mathbf{q}(1) \leftrightharpoons \mathbf{q}(j_0) $\;
			
		\For{$ k = 1, 2, \dots $}
		{
			$ \mathbf{d} = \mathbf{A} \bm{\ell}_k $, $ \bm{\ell} = \mathbf{B} \mathbf{d}_k $\;
			
			\For{$ i = 1, 2, \dots, k $}
			{
				$ h_{i, k} = \mathbf{d}(\mathbf{p}(i)) $, $ f_{i, k} = \bm{\ell}(\mathbf{q}(i)) $\;
				$ \mathbf{d} = \mathbf{d} - h_{i, k} \mathbf{d}_i $\;
				$ \bm{\ell} = \bm{\ell} - f_{i, k} \bm{\ell}_i $\;
			}
			Determine $ i_0 $ and $ j_0 $ such that $ \abs{d_{i_0}} = \max_{k+1 \le i \le m} \abs{\mathbf{d}(\mathbf{p}(i))} $ and $ \abs{\ell_{j_0}} = \max_{k+1 \le j \le n} \abs{\bm{\ell}(\mathbf{q}(j))} $\;
				
			$ h_{k+1, k} = d_{i_0} $, $ \mathbf{d}_{k+1} = \mathbf{d} / h_{k+1, k} $\;
			$ f_{k+1, k} = \ell_{j_0} $, $ \bm{\ell}_{k+1} = \bm{\ell} / f_{k+1, k} $\;
			$ \mathbf{p}(k+1) \leftrightharpoons \mathbf{p}(i_0) $, $ \mathbf{q}(k+1) \leftrightharpoons \mathbf{q}(j_0) $\;
		}
	\end{algorithm}
	
	From the above discussions, we have the relations \beq\label{mrelations}			\mathbf{A} \mathbf{L}_k = \mathbf{D}_{k+1} \mathbf{H}_{k+1, k}, \qquad
			\mathbf{B} \mathbf{D}_k = \mathbf{L}_{k+1} \mathbf{F}_{k+1, k}.\eeq If \Cref{alg:pfreehess} terminates (i.e., $ h_{k+1, k} = 0 $ or $ f_{k+1, k} = 0 $), the pivoting strategy ensures that $ \mathbf{d}_{k+1} = \mathbf{0} $ or $ \bm{\ell}_{k+1} = \mathbf{0} $.
	
	\begin{remark}
		Notice that there are two equivalent representations for the basis of the simultaneous Hessenberg process, i.e., \eqref{eq:repdk} and \eqref{eq:uppdk}. Numerical results in \Cref{ex:cond} demonstrates that the implementation based on \eqref{eq:uppdk} is better than that based on \eqref{eq:repdk}. %This is analogous to the relation between the modified Gram--Schmidt process and the classical Gram--Schmidt process. 
	\end{remark}
	
	\begin{example} \label{ex:cond}
		Consider the MATLAB test matrix {\tt Lotkin} which is ill-conditioned and has many negative eigenvalues of small magnitude. We generate $ \mathbf{A} $, $ \mathbf{B} $, $ \mathbf{b} $, and $ \mathbf{c} $ by the following MATLAB's scripts:
		\begin{verbatim}
			m = 1000; n = 1000; 
			A = gallery('lotkin', n); B = A'; 
			b = ones(m, 1); c = ones(n, 1);
		\end{verbatim}
		We plot the condition number of $ \mathbf{D}_k $ generated by \eqref{eq:repdk} and \eqref{eq:uppdk} for increasing $ k $ in \Cref{fig:cond}. We see that the condition number of $ \mathbf{D}_k $ generated by \eqref{eq:repdk} increases rapidly after $ k = 30 $, but that generated by \eqref{eq:uppdk} is nearly constant $(\approx 100)$ for $ k \geq 2 $.
		\end{example}
		
		\begin{figure}[H]
			\centering
			\includegraphics[height=2.25in]{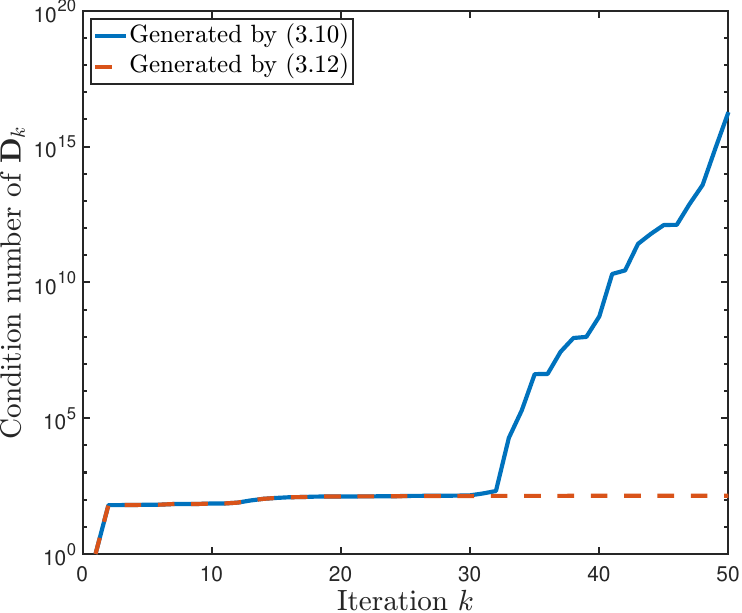}
			\caption{Condition number of $ \mathbf{D}_k $ generated by \eqref{eq:repdk} and \eqref{eq:uppdk}.}
			\label{fig:cond}
		\end{figure}

	\section{GP-CMRH} \label{sec:gpcmrh}
	We now introduce GP-CMRH for solving \eqref{eq:sys}. The derivation of GP-CMRH parallels that of GPMR. GP-CMRH is based on the simultaneous Hessenberg process, whereas GPMR is based on the orthogonal Hessenberg reduction. 
	
	Let $\mbf d_k$ and $\bm{\ell}_k$ be generated in \Cref{alg:pfreehess}. Let
	\[
		\mathbf{K}_0 = 
		\begin{bmatrix}
			\mathbf{0} & \mathbf{A} \\
			\mathbf{B} & \mathbf{0}
		\end{bmatrix}, 	
	\]
	and let 
	\[
		\bm{\Pi}_k = 
		\begin{bmatrix}
			\mathbf{e}_1 & \mathbf{e}_{k+1} & \cdots & \mathbf{e}_i & \mathbf{e}_{k+i} & \cdots & \mathbf{e}_k & \mathbf{e}_{2k}
		\end{bmatrix} \in \mathbb{R}^{2k \times 2k}
	\]
	be the permutation matrix introduced by Paige \cite{paige1974bidiagonalization}. Define 
	\begin{equation} \label{eq:Wk}
		\mathbf{W}_k := 
		\begin{bmatrix}
			\mathbf{D}_k \\ & \mathbf{L}_k
		\end{bmatrix} \bm{\Pi}_k =
		\begin{bmatrix}
			\mathbf{w}_1 & \mathbf{w}_2 & \cdots & \mathbf{w}_k
		\end{bmatrix}, \quad
		\mathbf{w}_k = 
		\begin{bmatrix}
			\mathbf{d}_k & \mathbf{0} \\ \mathbf{0} & \bm{\ell}_k
		\end{bmatrix}.
	\end{equation}
	Combining \eqref{mrelations} and \eqref{eq:Wk} yields
	\begin{equation} \label{eq:K0W}
		\begin{split}
			\mathbf{K}_0 \mathbf{W}_k = 
			\begin{bmatrix}
				\mathbf{D}_{k+1} \\ & \mathbf{L}_{k+1}
			\end{bmatrix} \bm{\Pi}_{k+1} \bm{\Pi}_{k+1}^\top
			\begin{bmatrix}
				& \mathbf{H}_{k+1, k} \\ \mathbf{F}_{k+1, k}
			\end{bmatrix} \bm{\Pi}_k =: \mathbf{W}_{k+1} \mathbf{G}_{k+1, k},
		\end{split}
	\end{equation}
	where 
	\[
		\mathbf{G}_{k+1, k} = 
		\begin{bmatrix}
			\bm{\Phi}_{1,1} & \bm{\Phi}_{1,2} & \cdots & \bm{\Phi}_{1,k} \\
			\bm{\Phi}_{2,1} & \bm{\Phi}_{2,2} & \ddots & \vdots \\
			& \ddots & \ddots & \bm{\Phi}_{k-1,k} \\
			& & \ddots & \bm{\Phi}_{k,k} \\
			& & & \bm{\Phi}_{k+1,k}
		\end{bmatrix}, \quad
		\bm{\Phi}_{i,j} = 
		\begin{bmatrix}
			0 & h_{i,j} \\ f_{i, j} & 0
		\end{bmatrix}.
	\]

Define
	\[
		\bm{\Lambda} := 
		\begin{bmatrix}
			\lambda \mathbf{I} \\ & \mu \mathbf{I}
		\end{bmatrix}, \quad
		\mathbf{K} := \mathbf{K}_0 + \bm{\Lambda} = 
		\begin{bmatrix}
			\lambda \mathbf{I} & \mathbf{A} \\
			\mathbf{B} & \mu \mathbf{I} 
		\end{bmatrix}, \quad
		\mathbf{g} := 
		\begin{bmatrix}
			\mathbf{b} \\ \mathbf{c}
		\end{bmatrix}.
	\]
	Since 
	\[
		\bm{\Lambda} \mathbf{w}_k = 
		\begin{bmatrix}
			\lambda \mathbf{I} \\ & \mu \mathbf{I}
		\end{bmatrix} \mathbf{w}_k = \mathbf{w}_k
		\begin{bmatrix}
			\lambda \\ & \mu
		\end{bmatrix},
	\]
	from \eqref{eq:K0W}, we obtain
	\begin{equation} \label{eq:KW}
		\mathbf{K} \mathbf{W}_k = (\mathbf{K}_0 + \bm{\Lambda}) \mathbf{W}_k = \mathbf{W}_{k+1} \mathbf{S}_{k+1, k},
	\end{equation}
	where
	\[
		\mathbf{S}_{k+1, k} := 
		\begin{bmatrix}
			\bm{\Theta}_{1,1} & \bm{\Phi}_{1,2} & \cdots & \bm{\Phi}_{1,k} \\
			\bm{\Phi}_{2,1} & \bm{\Theta}_{2,2} & \ddots & \vdots \\
			& \ddots & \ddots & \bm{\Phi}_{k-1,k} \\
			& & \ddots & \bm{\Theta}_{k,k} \\
			& & & \bm{\Phi}_{k+1,k}
		\end{bmatrix}, \quad
		\bm{\Theta}_{j,j} = \bm{\Phi}_{j, j} + 
		\begin{bmatrix}
			\lambda \\ & \mu
		\end{bmatrix} = 
		\begin{bmatrix}
			\lambda & h_{j,j} \\ f_{j, j} & \mu
		\end{bmatrix}.
	\]
	At step $ k $, GP-CMRH seeks the $ k $th iterate
	\[
		\begin{bmatrix}
			\mathbf{x}_k \\ \mathbf{y}_k
		\end{bmatrix}
		 = \mathbf{W}_k \mathbf{z}_k,
	\] where $\mathbf{z}_k \in \mathbb{R}^{2k}$ solves the minimization problem
	\begin{equation} \label{eq:zk}
		\mathbf{z}_k = \argmin_{\mathbf{z} \in \mathbb{R}^{2k}} \norm{\beta \mathbf{e}_1 + \gamma \mathbf{e}_2 - \mathbf{S}_{k+1, k} \mathbf{z}}.
	\end{equation} Here $ \beta $ and $ \gamma $ are determined by \Cref{alg:pfreehess}. 
	It follows from \eqref{eq:KW} that the corresponding residual 
	\[
		\mathbf{r}_k = 
		\begin{bmatrix}
			\mathbf{b} \\ \mathbf{c}
		\end{bmatrix} - 
		\begin{bmatrix}
			\lambda \mathbf{I} & \mathbf{A} \\
			\mathbf{B} & \mu \mathbf{I}
		\end{bmatrix}
		\begin{bmatrix}
			\mathbf{x}_k \\ \mathbf{y}_k
		\end{bmatrix} 
		= \mathbf{g} - \mathbf{K} \mathbf{W}_k \mathbf{z}_k 
		= \mathbf{W}_{k+1} (\beta \mathbf{e}_1 + \gamma \mathbf{e}_2 - \mathbf{S}_{k+1, k} \mathbf{z}_k).
	\]
 Since $ \mathbf{W}_k $ is not orthogonal, GP-CMRH is a quasi-minimal residual method. 
 
 The minimization problem \eqref{eq:zk} can be solved via the QR factorization of $ \mathbf{S}_{k+1, k} $. Considering the block upper Hessenberg structure of $ \mathbf{S}_{k+1, k} $, we use Givens rotations to obtain its QR factorization.  Let
	\[
		\mathbf{S}_{k+1, k} = \mathbf{Q}_k
		\begin{bmatrix}
			\mathbf{R}_k \\ \mathbf{0}
		\end{bmatrix}, \quad 
		\mathbf{R}_k \in \mathbb{R}^{2k \times 2k}
	\]
	be the reduced QR factorization, where 
	\[
		\mathbf{Q}_k^\top = \mathbf{Q}_{2k-1,2k+2} \cdots \mathbf{Q}_{3,6} \mathbf{Q}_{1,4} \in \mathbb{R}^{(2k+2) \times (2k+2)}.
	\]
	For $i = 1, \dots, k $, the structure of $ \mathbf{Q}_{2i-1, 2i+2} $ is 
	\[
		\mathbf{Q}_{2i-1, 2i+2} = 
		\begin{bmatrix}
			\mathbf{I}_{2i-2} \\
			& \times & \times & \times & \times \\			
			& \times & \times & \times & \times \\
			& \times & \times & \times & \times \\
			& \times & \times & \times & \times \\
			&        &		  &        &        & \mathbf{I}_{2k-2i}
		\end{bmatrix} \in \mathbb{R}^{(2k+2) \times (2k+2)},
	\]
	in which the $ 4 \times 4 $ block is the product of the following four Givens rotations:
	\[
		\begin{bmatrix}
			1 \\
			& c_{4,i} & s_{4,i} \\
			& -s_{4,i} & c_{4,i} \\
			& & & 1
		\end{bmatrix}
		\begin{bmatrix}
			1 \\
			& c_{3,i} & & s_{3,i} \\
			& & 1 \\
			& -s_{3,i} & & c_{3,i}
		\end{bmatrix}
		\begin{bmatrix}
			c_{2,i} & s_{2,i} \\
			-s_{2,i} & c_{2,i} \\
			& & 1 \\
			& & & 1
		\end{bmatrix}
		\begin{bmatrix}
			c_{1,i} & & & s_{1,i} \\
			        & 1 \\
			        & & 1 \\
			-s_{1,i} & & & c_{1,i}
		\end{bmatrix}.
	\]
	The result $ \begin{bmatrix} y_1 & y_2 & y_3 & y_4 \end{bmatrix}^\top $ of a matrix-vector product between the above four Givens rotations and a vector $ \begin{bmatrix} x_1 & x_2 & x_3 & x_4 \end{bmatrix}^\top $ can be obtained via \Cref{alg:givens}.
	
	\begin{algorithm}[htbp]
%		\setstretch{1.15}
		\caption{Subroutine: \texttt{givens}}
		\label{alg:givens}
		
		\KwIn{$ i $, $ x_1 $, $ x_2 $, $ x_3 $, $ x_4 $ }
		\KwOut{$ y_1 $, $ y_2 $, $ y_3 $, $ y_4 $}
		
		$ t = c_{1,i}x_1 + s_{1,i}x_4 $, $ y_4 = c_{1,i}x_4 - s_{1,i}x_1 $, $ y_1 = t $ \tcc*[r]{First Givens rotation}
		
		$ t = c_{2,i}y_1 + s_{2,i}x_2 $, $ y_2 = c_{2,i}x_2 - s_{2,i}y_1 $, $ y_1 = t $ \tcc*[r]{Second Givens rotation}
		
		$ t = c_{3,i}y_2 + s_{3,i}y_4 $, $ y_4 = c_{3,i}y_4 - s_{3,i}y_2 $, $ y_2 = t $ \tcc*[r]{Third Givens rotation}
		
		$ t = c_{4,i}y_2 + s_{4,i}x_3 $, $ y_3 = c_{4,i}x_3 - s_{4,i}y_2 $, $ y_2 = t $ \tcc*[r]{Forth Givens rotation}
	\end{algorithm} 
	
	At step $ k $, we need to apply previous Givens rotations and compute $ \mathbf{Q}_{2k-1,2k+2} $ to update $ \mathbf{R}_k $, that is,
	\[
		\mathbf{Q}_{2k-1,2k+2} \mathbf{Q}_{k-1}^\top 
		\begin{bmatrix}
			\bm{\Phi}_{1,k} \\ \vdots \\ \bm{\Theta}_{k,k} \\ \bm{\Phi}_{k+1,k}
		\end{bmatrix}
		= \mathbf{Q}_{2k-1,2k+2} 
		\begin{bmatrix}
			r_{1,2 k-1} & r_{1,2 k} \\
			\vdots & \vdots \\
			r_{2 k-2,2 k-1} & r_{2 k-2,2 k} \\
			\widetilde{r}_{2 k-1,2 k-1} & \widetilde{r}_{2 k-1,2 k} \\
			\widetilde{r}_{2 k, 2 k-1} & \widetilde{r}_{2 k, 2 k} \\
			& h_{k+1, k} \\
			f_{k+1,k}
		\end{bmatrix}
		= \begin{bmatrix}
			r_{1,2 k-1} & r_{1,2 k} \\
			\vdots & \vdots \\
			r_{2 k-2,2 k-1} & r_{2 k-2,2 k} \\
			r_{2 k-1,2 k-1} & r_{2 k-1,2 k} \\
			0 & r_{2 k, 2 k} \\
			0 & 0 \\
			0 & 0
		\end{bmatrix}.
	\]
	The implementation details for this update are presented in \Cref{alg:qr}.
	
	\begin{algorithm}[htbp]
		\caption{Subroutine: \texttt{qr}}
		\label{alg:qr}
		
		\KwIn{$ k $, $ \widetilde{r}_{2 k-1,2 k-1}$, $ \widetilde{r}_{2 k-1,2 k} $, $ \widetilde{r}_{2 k, 2 k-1} $, $ \widetilde{r}_{2 k, 2 k} $, $ h_{k+1, k} $, $f_{k+1,k} $}
		\KwOut{$ c_{1,k} $, $ c_{2,k} $, $ c_{3,k} $, $ c_{4,k} $, $ s_{1,k} $, $ s_{2,k} $, $ s_{3,k} $, $ s_{4,k} $, $ r_{2k-1, 2k-1} $, $ r_{2k-1, 2k} $, $ r_{2k, 2k} $}
		
		\tcc*[l]{Compute first Givens rotation, and zero out $ f_{k+1,k} $}
		$ \widehat{r}_{2k-1,2k-1} = \paren{ \widetilde{r}_{2k-1,2k-1}^2 + f_{k+1,k}^2 }^{\frac{1}{2}} $ \;
		$ c_{1,k} = \widetilde{r}_{2k-1,2k-1}/\widehat{r}_{2k-1,2k-1} $, $ s_{1,k} = f_{k+1,k}/\widehat{r}_{2k-1,2k-1} $\;
		$ \widehat{r}_{2k-1,2k} = c_{1,k} \widetilde{r}_{2k-1,2k} $, $ \widehat{r}_{2k+2,2k} = -s_{1,k} \widetilde{r}_{2k-1,2k} $\;
		
		\tcc*[l]{Compute second Givens rotation, and zero out $ \widetilde{r}_{2k,2k-1} $}
		$ r_{2k-1,2k-1} = \paren{ \widehat{r}_{2k-1,2k-1}^2 + \widetilde{r}_{2k,2k-1} }^{\frac{1}{2}} $\; 
		$ c_{2,k} = \widehat{r}_{2k-1,2k-1}/r_{2k-1,2k-1} $, $ s_{2,k } = \widetilde{r}_{2k,2k-1} / r_{2k-1,2k-1} $\;
		$ r_{2k-1, 2k} = c_{2,k}\widehat{r}_{2k-1,2k} + s_{2,k}\widetilde{r}_{2k,2k} $\;
		$ \widehat{r}_{2k,2k} = c_{2,k}\widetilde{r}_{2k,2k} - s_{2,k}\widehat{r}_{2k-1,2k} $\;
		
		\tcc*[l]{Compute third Givens rotation, and zero out $ \widehat{r}_{2k+2, 2k} $}
		$ \mathring{r}_{2k,2k} = \paren{ \widehat{r}_{2k,2k}^2 + \widehat{r}_{2k+2,2k} }^{\frac{1}{2}} $\;
		$ c_{3,k} = \widehat{r}_{2k,2k}/\mathring{r}_{2k,2k} $, $ s_{3,k} = \widehat{r}_{2k+2,2k}/\mathring{r}_{2k,2k} $\;
		
		\tcc*[l]{Compute forth Givens rotation, and zero out $ h_{k+1,k} $}
		$ r_{2k,2k} = \paren{ \mathring{r}_{2k,2k}^2 + h_{k+1,k}^2 }^{\frac{1}{2}} $\;
		$ c_{4,k} = \mathring{r}/r_{2k,2k} $, $ s_{4,k} = h_{k+1,k}/r_{2k,2k} $\;
	\end{algorithm}
	Define
	\[
		\wtilde{\mathbf{t}}_k := \mathbf{Q}_k^\top (\beta \mathbf{e}_1 + \gamma \mathbf{e}_2) = 
		\begin{bmatrix}
			\tau_1 & \cdots & \tau_{2k} & \wtilde{\tau}_{2k+1} & \wtilde{\tau}_{2k+2}
		\end{bmatrix} =:
		\begin{bmatrix}
			\mathbf{t}_k & \wtilde{\tau}_{2k+1} & \wtilde{\tau}_{2k+2}
		\end{bmatrix} \in \mathbb{R}^{2k+2}.
	\]
	Then, the vector $ \mathbf{z}_k $ is given by
	\[
		\mathbf{z}_k = \mathbf{R}_k^{-1} \mathbf{t}_k,
	\]
	and the estimation of the residual norm is 
	\begin{equation} \label{eq:nrk}
		\begin{split}
			\norm{\mathbf{r}_k} &= \norm{\mathbf{W}_{k+1} (\beta \mathbf{e}_1 + \gamma \mathbf{e}_2 - \mathbf{S}_{k+1, k} \mathbf{z}_k)} 
			= \norm[\bigg]{\mathbf{W}_{k+1} \mathbf{Q}_k \biggl(\wtilde{\mathbf{t}}_k - 
			\begin{bmatrix}
				\mathbf{R}_k \\ \mathbf{0}
			\end{bmatrix} \mathbf{z}_k \biggr)} \\
			&\le \norm{\mathbf{W}_{k+1}} \sqrt{\wtilde{\tau}_{2k+1}^2 + \wtilde{\tau}_{2k+2}^2} \\
			&= \max \{\norm{\mathbf{D}_{k+1}}, \norm{\mathbf{L}_{k+1}}\} \sqrt{\wtilde{\tau}_{2k+1}^2 + \wtilde{\tau}_{2k+2}^2}  \\
			&\le \max \{ \norm{\mathbf{D}_{k+1}}_{\rmf}, \norm{\mathbf{L}_{k+1}}_{\rmf} \} \sqrt{\wtilde{\tau}_{2k+1}^2 + \wtilde{\tau}_{2k+2}^2} \\
			&\le \sqrt{\frac{(2\max \{m, n\} - k)(k+1)}{2}} \sqrt{\wtilde{\tau}_{2k+1}^2 + \wtilde{\tau}_{2k+2}^2}.
		\end{split}
	\end{equation}
	The last inequality follows from the sparsity patterns of $ \mathbf{D}_k $ and $ \mathbf{L}_k $, and $ \abs{\mathbf{e}_i^\top \mathbf{D}_k \mathbf{e}_j} \le 1 $ and $ \abs{\mathbf{e}_i^\top \mathbf{L}_k \mathbf{e}_j} \le 1 $ guaranteed by the pivoting strategy. 
	
	We summarize the implementation details of GP-CMRH in \Cref{alg:gpcmrh}.
	
	\begin{algorithm}[htbp]
%		\setstretch{1.15}
		\caption{GP-CMRH}
		\label{alg:gpcmrh}
		
		\KwIn{$ \mathbf{A} \in \mathbb{R}^{m \times n} $, $ \mathbf{B} \in \mathbb{R}^{n \times m} $, nonzero $ \mathbf{b} \in \mathbb{R}^m $ and $ \mathbf{c} \in \mathbb{R}^n $, $ \lambda,\ \mu \in \mathbb{R} $, maximum number of iterations \texttt{maxit}, and tolerance \texttt{tol} }
			
		$ \mathbf{p} = \begin{bmatrix} 1 & 2 & \cdots & m \end{bmatrix}^\top $, $ \mathbf{q} = \begin{bmatrix} 1 & 2 & \cdots & n \end{bmatrix}^\top $\;
			
		Determine $ i_0 $ and $ j_0 $ such that $ \abs{b_{i_0}} = \max_{1 \le i \le m} \abs{\mathbf{e}_i^\top \mathbf{b}} $ and $ \abs{c_{j_0}} = \max_{1 \le j \le n} \abs{\mathbf{e}_j^\top \mathbf{c}} $\;
			
		$ \beta = b_{i_0} $, $ \mathbf{d}_1 = \mathbf{b} / \beta $\;
		$ \gamma = c_{j_0} $, $ \bm{\ell}_1 = \mathbf{c} / \gamma $\;
		$ \mathbf{p}(1) \leftrightharpoons \mathbf{p}(i_0) $, $ \mathbf{q}(1) \leftrightharpoons \mathbf{q}(j_0) $\;
		$ \widetilde{\tau}_{1} = \beta,\ \widetilde{\tau}_{2} = \gamma $\;
			
		\For{$ k = 1, 2, \dots, \mathtt{maxit} $}
		{
			$ \mathbf{d} = \mathbf{A} \bm{\ell}_k $, $ \bm{\ell} = \mathbf{B} \mathbf{d}_k $\;
				
			\For{$ i = 1, 2, \dots, k $}
			{
				$ h_{i, k} = \mathbf{d}(\mathbf{p}(i)) $, $ f_{i, k} = \bm{\ell}(\mathbf{q}(i)) $\;
				$ \mathbf{d} = \mathbf{d} - h_{i, k} \mathbf{d}_i $\;
				$ \bm{\ell} = \bm{\ell} - f_{i, k} \bm{\ell}_i $\;
				Update $ \mathbf{S}_{2i-1, 2k} = h_{i, k} $ and $ \mathbf{S}_{2i, 2k-1} = f_{i, k} $\;
			}
			Determine $ i_0 $ and $ j_0 $ such that $ \abs{d_{i_0}} = \max_{k+1 \le i \le m} \abs{\mathbf{d}(\mathbf{p}(i))} $ and $ \abs{\ell_{j_0}} = \max_{k+1 \le j \le n} \abs{\bm{\ell}(\mathbf{q}(j))} $\;
			$ h_{k+1, k} = d_{i_0} $, $ f_{k+1, k} = \ell_{j_0} $\;
			
			$ \widetilde{r}_{1,2k} = h_{1,k},\ \widetilde{r}_{2,2k-1} = f_{1,k} $\;
			\leIf{$ k = 1$}{$ \widetilde{r}_{1,2k-1} = \lambda,\  \widetilde{r}_{2,2k} = \mu $}{$ \widetilde{r}_{1,2k-1} =  \widetilde{r}_{2,2k} = 0 $}
			
			\For(\tcc*[f]{Apply previous Givens rotations}){$ j = 1,2,\dotsc, k-1 $}
			{
				\leIf{$ j = k-1 $}{$ \rho = \lambda,\ \delta = \mu $}{$ \rho = \delta = 0 $}
				
				$ r_{2j-1,2k-1},\ r_{2j,2k-1},\ \widetilde{r}_{2j+1,2k-1},\ \widetilde{r}_{2j+2,2k-1} = \mathtt{givens}(j, \widetilde{r}_{2j-1,2k-1}, \widetilde{r}_{2j,2k-1}, \rho, f_{j+1,k}) $\;
				
				$ r_{2j-1,2k},\ r_{2j,2k},\ \widetilde{r}_{2j+1,2k},\ \widetilde{r}_{2j+2,2k} = \mathtt{givens}(j, \widetilde{r}_{2j-1,2k}, \widetilde{r}_{2j,2k}, h_{j+1,k}, \delta) $\;
			}
			
			\tcc*[r]{Compute $ Q_{2k-1,2k+2} $ to update $ \mathbf{R}_k $}
			$ c_{1,k} $, $ c_{2,k} $, $ c_{3,k} $, $ c_{4,k} $, $ s_{1,k} $, $ s_{2,k} $, $ s_{3,k} $, $ s_{4,k} $, $ r_{2k-1, 2k-1} $, $ r_{2k-1, 2k} $, $ r_{2k, 2k} = \mathtt{qr}(k, \widetilde{r}_{2 k-1,2 k-1}, \widetilde{r}_{2 k-1,2 k}, \widetilde{r}_{2 k, 2 k-1}, \widetilde{r}_{2 k, 2 k}, h_{k+1, k}, f_{k+1,k}) $\;
			
			$ \tau_{2k-1},\ \tau_{2k},\ \widetilde{\tau}_{2k+1},\ \widetilde{\tau}_{2k+2} = \mathtt{givens}(\widetilde{\tau}_{2k-1},\ \widetilde{\tau}_{2k}, 0, 0) $ \tcc*[r]{Update $ \mathbf{t}_k $}
			
			Compute estimation of residual norm $ \rho_k $ via \eqref{eq:nrk}\;
				
			\If{$ \rho_k \le \mathtt{tol} $}
			{
					\textbf{break}\;
			}
				
			$ \mathbf{d}_{k+1} = \mathbf{d} / h_{k+1, k} $, $ \bm{\ell}_{k+1} = \bm{\ell} / f_{k+1, k} $\;
			$ \mathbf{p}(k+1) \leftrightharpoons \mathbf{p}(i_0) $, $ \mathbf{q}(k+1) \leftrightharpoons \mathbf{q}(j_0) $\;
		}
			
		$ \mathbf{z}_k = \mathbf{R}_k^{-1} \mathbf{t}_k $, $ \mathbf{x}_k = \sum_{i=1}^{k} \mathbf{e}_{2i-1}^\top \mathbf{z}_k \mathbf{d}_i $, $ \mathbf{y}_k = \sum_{i=1}^{k} \mathbf{e}_{2i}^\top \mathbf{z}_k \bm{\ell}_i $\;
	\end{algorithm}
	
	\subsection{Relation between GP-CMRH and GPMR}
	For the sake of clarity, the superscripts $ \cdot^\mathrm{GP\mbox{\footnotesize -}CMRH} $ and $ \cdot^\mathrm{GPMR} $ are corresponding to GP-CMRH and GPMR, respectively. GPMR seeks the approximate solution $ \begin{bmatrix}
	\mathbf{x}_k^\mathrm{GPMR} \\ \mathbf{y}_k^\mathrm{GPMR}
	\end{bmatrix} \in \range(\mathbf{W}_k) $ that minimizes the norm of the corresponding residual. More precisely,
    \[
    	{\renewcommand{\arraystretch}{1.12}
    	\begin{bmatrix}
    		\mathbf{x}_k^\mathrm{GPMR} \\ \mathbf{y}_k^\mathrm{GPMR}
    	\end{bmatrix} = \mathbf{W}_k \mathbf{z}_k^\mathrm{GPMR},}
    \]
    and $ \mathbf{z}_k^\mathrm{GPMR} \in\mbbr^{2k}$ solves 
    \[
		\min_{\mathbf{z} \in \mathbb{R}^{2k}} \norm{\mathbf{g} - \mathbf{K} \mathbf{W}_k \mathbf{z}} = \min_{\mathbf{z} \in \mathbb{R}^{2k}} \norm{\mathbf{W}_{k+1} (\beta \mathbf{e}_1 + \gamma \mathbf{e}_2 - \mathbf{S}_{k+1, k} \mathbf{z})}.
    \]
	We have the following result.
	
	\begin{theorem} \label{thm:rk_gpchmr_gpmr}
		Let $ \mathbf{r}_k^\mathrm{GP\mbox{\footnotesize -}CMRH} $ and $ \mathbf{r}_k^\mathrm{GPMR} $ be the $k$th residuals of {\rm GP-CMRH} and {\rm GPMR}, respectively. Then,
		\[
			\norm{\mathbf{r}_k^\mathrm{GPMR}} \le \norm{\mathbf{r}_k^\mathrm{GP\mbox{\footnotesize -}CMRH}} \le \kappa(\mathbf{W}_{k+1}) \norm{\mathbf{r}_k^\mathrm{GPMR}},
 		\]
 		where $ \kappa(\mathbf{W}_{k+1}) = \norm{\mathbf{W}_{k+1}} \|\mathbf{W}_{k+1}^\dagger\| $ is the condition number of $ \mathbf{W}_{k+1} $.
	\end{theorem}
	
	\begin{proof}
		Since the iterates of GP-CMRH and GPMR are in the same subspace and GPMR minimizes the corresponding residual norm, we have $ \norm{\mathbf{r}_k^\mathrm{GPMR}} \le \norm{\mathbf{r}_k^\mathrm{GP\mbox{\footnotesize -}CMRH}} $. We now prove the right inequality. From the optimality conditions of GP-CMRH and GPMR, we have 
		\begin{align*}
			\norm{\mathbf{r}_k^\mathrm{GP\mbox{\footnotesize -}CMRH}} &= \norm{\mathbf{W}_{k+1} \bigl(\beta \mathbf{e}_1 + \gamma \mathbf{e}_2 - \mathbf{S}_{k+1, k} \mathbf{z}_k^\mathrm{GP\mbox{\footnotesize -}CMRH} \bigr)} \\
			&\le \norm{\mathbf{W}_{k+1}} \norm{\beta \mathbf{e}_1 + \gamma \mathbf{e}_2 - \mathbf{S}_{k+1, k} \mathbf{z}_k^\mathrm{GP\mbox{\footnotesize -}CMRH}} \\
			&\le \norm{\mathbf{W}_{k+1}} \norm{\beta \mathbf{e}_1 + \gamma \mathbf{e}_2 - \mathbf{S}_{k+1, k} \mathbf{z}_k^\mathrm{GPMR}} \\
			&= \norm{\mathbf{W}_{k+1}} \norm{\mathbf{W}_{k+1}^\dagger \mathbf{W}_{k+1} (\beta \mathbf{e}_1 + \gamma \mathbf{e}_2 - \mathbf{S}_{k+1, k} \mathbf{z}_k^\mathrm{GPMR})} \\
			&\le \norm{\mathbf{W}_{k+1}} \lVert \mathbf{W}_{k+1}^\dagger \rVert \norm{\mathbf{r}_k^\mathrm{GPMR}} \\
			&= \kappa(\mathbf{W}_{k+1}) \norm{\mathbf{r}_k^\mathrm{GPMR}}. \qedhere
		\end{align*}
	\end{proof}
	
	\section{Numerical experiments} \label{sec:exp}
	
	In this section, we compare the numerical performances among the four iterative methods GPMR, GP-CMRH, GMRES, and CMRH. All the algorithms stop as soon as reaching the maximum number of iterations \texttt{maxit} or the relative residual norm $ \lVert \mathbf{r}_k \rVert / \lVert \mathbf{r}_0 \rVert \le \mathtt{tol} $. All experiments are performed using MATLAB R2024b on MacBook Air with Apple M3 chip, 16 GB memory, and macOS Sequoia 15.4.1. For all experiments, the initial guess is set to be the zero vector, and the right-hand side vector is generated so that the exact solution is the vector of ones. 
	
	We use some nonsymmetric matrices from the SuitSparse Matrix Collection \cite{davis2011university} to construct system \eqref{eq:gpsys} by utilizing the graph partitioning tool METIS, and apply the block diagonal right-preconditioner $ \diag(\mathbf{M}, \mathbf{N}) $ in the four iterative methods.
	We set the parameters $ \mathtt{tol} = 1 \times 10^{-10} $, and $ \mathtt{maxit} = 600 $. 
	We plot the convergence histories of problems {\tt powersim},  {\tt c-53}, {\tt copter2}, {\tt Goodwin\_071}, {\tt venkat50}, {\tt PR02R}, {\tt cont-300}, and {\tt thermomech\_dK} in \Cref{fig:real_data}. We report the numbers of iterations, the runtimes, and the relative residual norms of the four methods in \Cref{tab:real_data}. We have the following observations. \bit
	\item[(i)] The convergence behavior of GP-CMRH is comparable to that of GPMR in terms of the number of iterations required to reach the specified relative residual tolerance.
	\item[(ii)] In most cases, GP-CMRH requires less runtimes compared to GPMR. 
	\item[(iii)] GP-CMRH performs significantly better than GMRES and CMRH in terms of convergence rate and runtime efficiency.
	\eit

\section{Concluding remarks} \label{sec:conclusion}
	
	In this paper, we have proposed a new simultaneous Hessenberg process that reduces two rectangular matrices to upper Hessenberg form simultaneously, without employing inner products. Based on the simultaneous Hessenberg process, we have proposed the inner product free iterative method GP-CMRH for solving block two-by-two nonsymmetric linear systems. Our numerical experiments show that the convergence behavior of GP-CMRH is similar to that of GPMR, whereas requires less computational time in most cases. We emphasize that GP-CMRH is inner product free, and this feature may be useful for high performance computing and low or mixed precision arithmetic.
	
	Recently, some state-of-the-art methods apply randomized sketching techniques to accelerate standard Krylov subspace methods; see, e.g., \cite{balabanov2022randomized,guettel2023randomized,nakatsukasa2024fast,guttel2024sketc,cortinovis2024speed,palitta2025sketched, burke2025gmres,balabanov2025rando,jang2025rando}. Randomized variants of GP-CMRH will be a topic of the future work.

\begin{figure}[H]
\centerline{\epsfig{figure=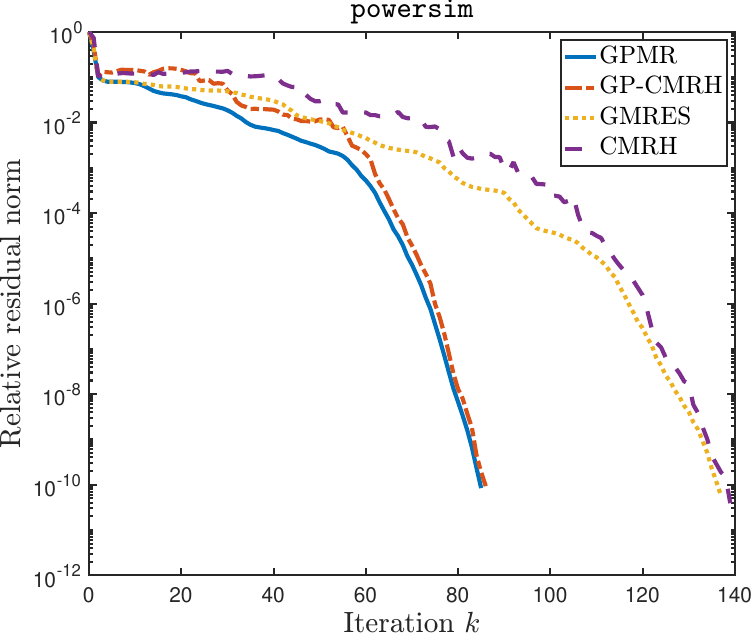,height=2.25in}\qquad\epsfig{figure=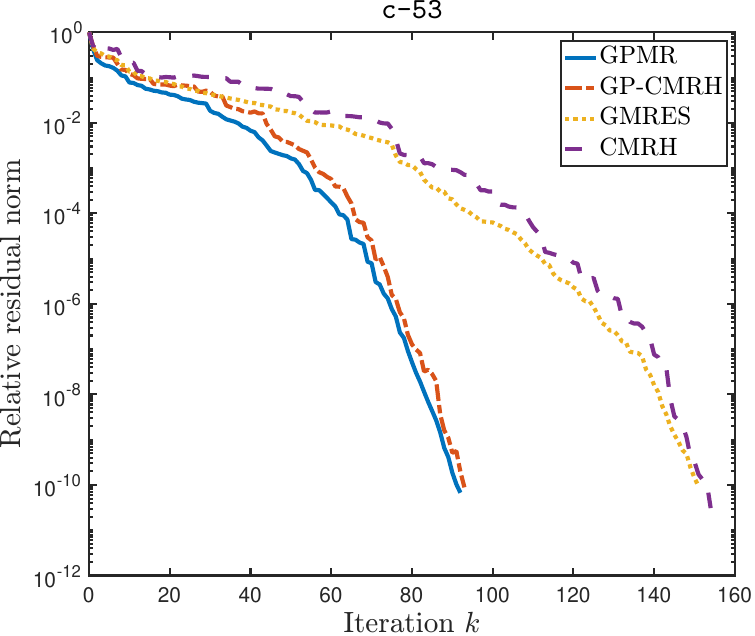,height=2.25in}}\vspace{3mm}

\centerline{\epsfig{figure=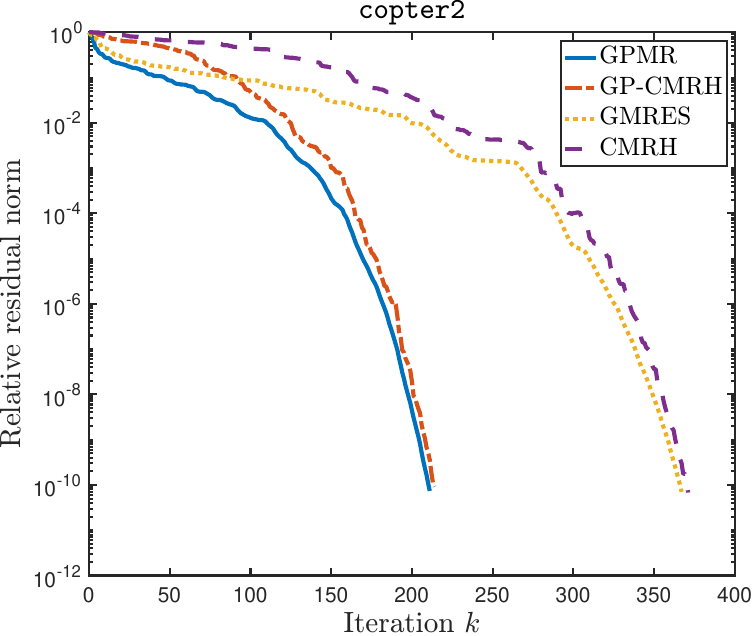,height=2.25in}\qquad\epsfig{figure=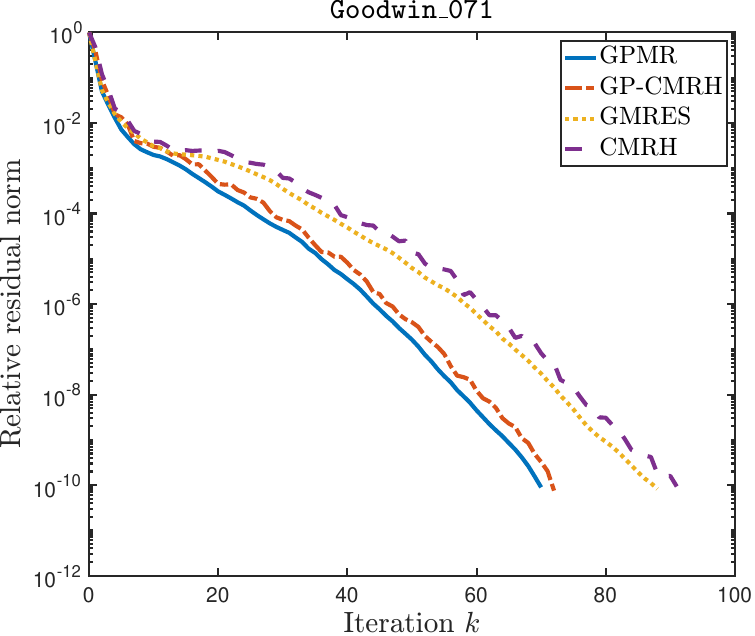,height=2.25in}}\vspace{3mm}

\centerline{\epsfig{figure=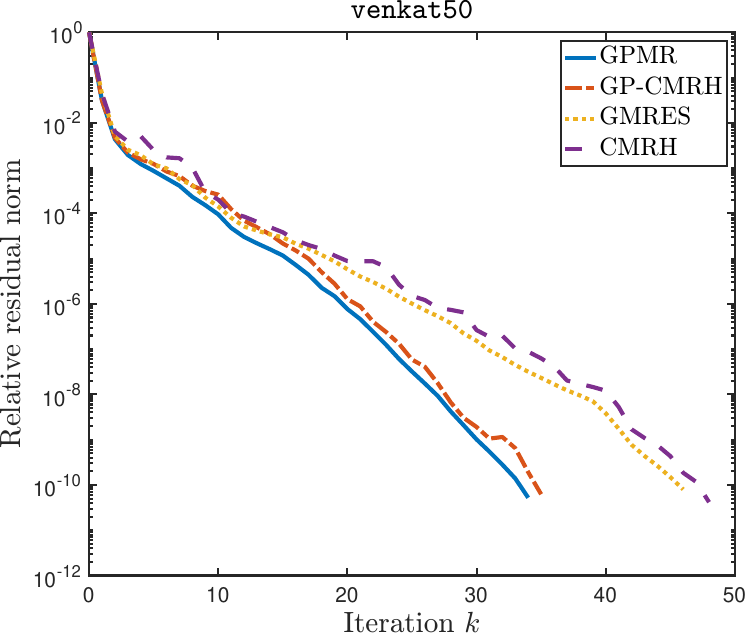,height=2.25in}\qquad\epsfig{figure=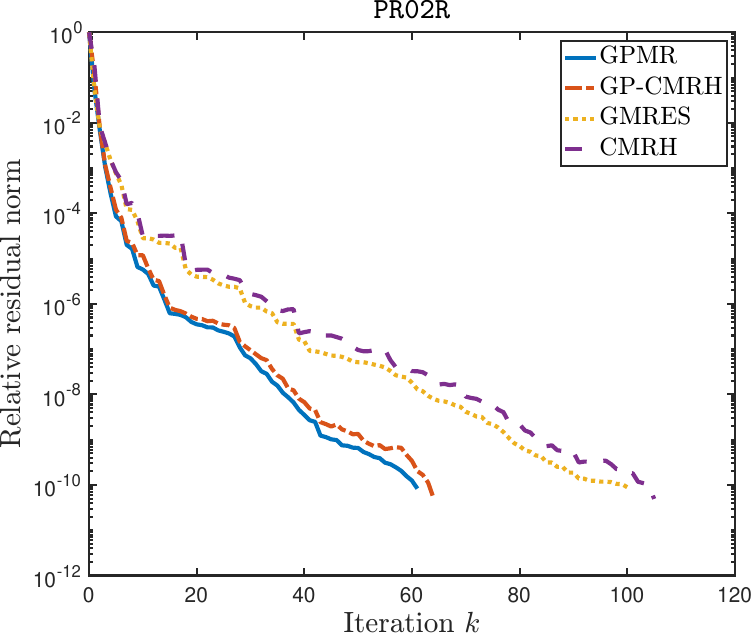,height=2.25in}}\vspace{3mm}

\centerline{\epsfig{figure=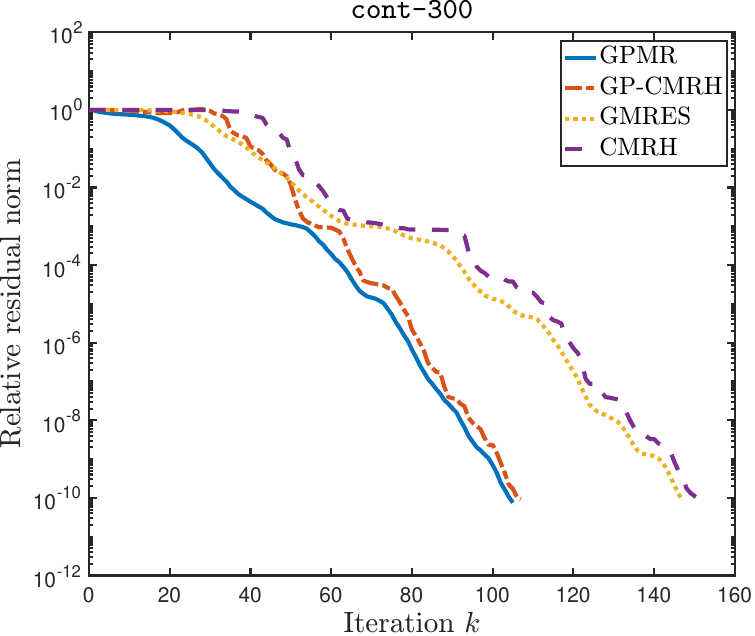,height=2.25in}\qquad\epsfig{figure=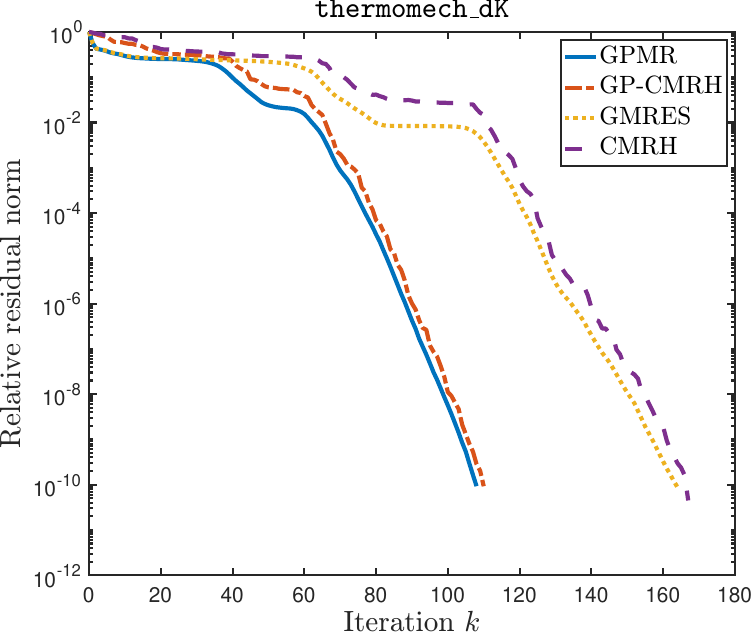,height=2.25in}}\caption{Convergence histories of GPMR, GP-CMRH, GMRES, and CMRH on problems {\tt powersim},  {\tt c-53}, {\tt copter2}, {\tt Goodwin\_071}, {\tt venkat50}, {\tt PR02R}, {\tt cont-300}, and {\tt thermomech\_dK}.}
		\label{fig:real_data}
	\end{figure}	
	
	\begin{landscape}
		\begin{table}
			\centering
			\caption{Numbers of iterations (Iter), runtimes (Time), and relative residual norms (Rel) of GPMR, GP-CMRH, GMRES, and CMRH on twenty-two matrices from the SuitSparse Matrix Collection. ``Nnz'' denotes the number of nonzero elements in each sparse matrix. Bold-faced values in the runtime column highlight the shortest time taken among the four methods.\vspace{4mm}}
			
			\label{tab:real_data}
			\begin{tabular}{|c|c|c|ccc|ccc|ccc|ccc|}
				\hline
				\multirow{2}{*}{Name} & \multirow{2}{*}{Size} & \multirow{2}{*}{Nnz} & \multicolumn{3}{c|}{GPMR} & \multicolumn{3}{c|}{GP-CMRH} & \multicolumn{3}{c|}{GMRES} & \multicolumn{3}{c|}{CMRH} \\ 
				\cline{4-15} 
				&  &  & \multicolumn{1}{c|}{Iter} & \multicolumn{1}{c|}{Time} & Rel & \multicolumn{1}{c|}{Iter} & \multicolumn{1}{c|}{Time} & Rel & \multicolumn{1}{c|}{Iter} & \multicolumn{1}{c|}{Time} & Rel & \multicolumn{1}{c|}{Iter} & \multicolumn{1}{c|}{Time} & Rel \\ 
				\hline
				\texttt{bcsstk17} & 10974 & 428650 & \multicolumn{1}{c|}{121} & \multicolumn{1}{c|}{0.39} & 3.66e-11 & \multicolumn{1}{c|}{121} & \multicolumn{1}{c|}{\textbf{0.28}} & 8.51e-11 & \multicolumn{1}{c|}{213} & \multicolumn{1}{c|}{1.57} & 8.17e-11 & \multicolumn{1}{c|}{216} & \multicolumn{1}{c|}{0.72} & 8.64e-11 \\ \hline
				\texttt{bcsstk25} & 15439 & 252241 & \multicolumn{1}{c|}{62} & \multicolumn{1}{c|}{0.18} & 5.54e-11 & \multicolumn{1}{c|}{63} & \multicolumn{1}{c|}{\textbf{0.15}} & 8.63e-11 & \multicolumn{1}{c|}{100} & \multicolumn{1}{c|}{0.47} & 8.13e-11 & \multicolumn{1}{c|}{116} & \multicolumn{1}{c|}{0.37} & 8.41e-11 \\ \hline
				\texttt{powersim} & 15838 & 64424 & \multicolumn{1}{c|}{85} & \multicolumn{1}{c|}{0.26} & 8.40e-11 & \multicolumn{1}{c|}{86} & \multicolumn{1}{c|}{\textbf{0.20}} & 9.28e-11 & \multicolumn{1}{c|}{137} & \multicolumn{1}{c|}{1.17} & 5.64e-11 & \multicolumn{1}{c|}{139} & \multicolumn{1}{c|}{0.42} & 3.91e-11 \\ \hline
				\texttt{raefsky3} & 21200 & 1488768 & \multicolumn{1}{c|}{37} & \multicolumn{1}{c|}{\textbf{0.68}} & 9.56e-11 & \multicolumn{1}{c|}{40} & \multicolumn{1}{c|}{0.69} & 8.86e-11 & \multicolumn{1}{c|}{63} & \multicolumn{1}{c|}{1.24} & 8.59e-11 & \multicolumn{1}{c|}{67} & \multicolumn{1}{c|}{1.15} & 9.28e-11 \\ \hline
				\texttt{sme3Db} & 29067 & 2081063 & \multicolumn{1}{c|}{65} & \multicolumn{1}{c|}{2.42} & 6.60e-11 & \multicolumn{1}{c|}{66} & \multicolumn{1}{c|}{\textbf{2.23}} & 7.30e-11 & \multicolumn{1}{c|}{97} & \multicolumn{1}{c|}{4.65} & 6.30e-11 & \multicolumn{1}{c|}{98} & \multicolumn{1}{c|}{3.27} & 9.40e-11 \\ \hline
				\texttt{c-53} & 30235 & 355139 & \multicolumn{1}{c|}{92} & \multicolumn{1}{c|}{1.38} & 6.73e-11 & \multicolumn{1}{c|}{93} & \multicolumn{1}{c|}{\textbf{1.21}} & 8.85e-11 & \multicolumn{1}{c|}{151} & \multicolumn{1}{c|}{3.29} & 9.34e-11 & \multicolumn{1}{c|}{154} & \multicolumn{1}{c|}{2.17} & 3.08e-11 \\ \hline
				\texttt{sme3Dc} & 42930 & 3148656 & \multicolumn{1}{c|}{97} & \multicolumn{1}{c|}{5.72} & 6.12e-11 & \multicolumn{1}{c|}{98} & \multicolumn{1}{c|}{\textbf{5.36}} & 7.67e-11 & \multicolumn{1}{c|}{161} & \multicolumn{1}{c|}{11.05} & 7.39e-11 & \multicolumn{1}{c|}{163} & \multicolumn{1}{c|}{8.75} & 6.85e-11 \\ \hline
				\texttt{bcsstk39} & 46772 & 2060662 & \multicolumn{1}{c|}{205} & \multicolumn{1}{c|}{5.72} & 7.54e-11 & \multicolumn{1}{c|}{209} & \multicolumn{1}{c|}{\textbf{3.25}} & 7.33e-11 & \multicolumn{1}{c|}{381} & \multicolumn{1}{c|}{23.32} & 9.73e-11 & \multicolumn{1}{c|}{392} & \multicolumn{1}{c|}{5.39} & 9.50e-11 \\ \hline
				{\tt rma10} & 46835 & 2329092 & \multicolumn{1}{c|}{41} & \multicolumn{1}{c|}{1.43} & 6.39e-11 & \multicolumn{1}{c|}{42} & \multicolumn{1}{c|}{\textbf{1.33}} & 6.34e-11 & \multicolumn{1}{c|}{49} & \multicolumn{1}{c|}{1.69} & 7.02e-11 & \multicolumn{1}{c|}{51} & \multicolumn{1}{c|}{1.53} & 5.08e-11 \\ \hline
				{\tt copter2} & 55476 & 759952 & \multicolumn{1}{c|}{211} & \multicolumn{1}{c|}{19.86} & 7.38e-11 & \multicolumn{1}{c|}{214} & \multicolumn{1}{c|}{\textbf{16.11}} & 9.18e-11 & \multicolumn{1}{c|}{367} & \multicolumn{1}{c|}{50.06} & 7.04e-11 & \multicolumn{1}{c|}{371} & \multicolumn{1}{c|}{27.06} & 6.81e-11 \\ \hline
				\texttt{Goodwin\_071} & 56021 & 1797934 & \multicolumn{1}{c|}{70} & \multicolumn{1}{c|}{2.77} & 8.93e-11 & \multicolumn{1}{c|}{72} & \multicolumn{1}{c|}{\textbf{2.39}} & 7.56e-11 & \multicolumn{1}{c|}{88} & \multicolumn{1}{c|}{4.24} & 8.20e-11 & \multicolumn{1}{c|}{91} & \multicolumn{1}{c|}{2.94} & 9.34e-11 \\ \hline
				\texttt{water\_tank} & 60740 & 2035281 & \multicolumn{1}{c|}{324} & \multicolumn{1}{c|}{46.00} & 8.07e-11 & \multicolumn{1}{c|}{338} & \multicolumn{1}{c|}{\textbf{35.09}} & 7.20e-11 & \multicolumn{1}{c|}{430} & \multicolumn{1}{c|}{75.59} & 9.73e-11 & \multicolumn{1}{c|}{464} & \multicolumn{1}{c|}{51.15} & 8.34e-11 \\ \hline
				\texttt{venkat50} & 62424 & 1717777 & \multicolumn{1}{c|}{34} & \multicolumn{1}{c|}{1.06} & 5.23e-11 & \multicolumn{1}{c|}{35} & \multicolumn{1}{c|}{\textbf{0.97}} & 6.24e-11 & \multicolumn{1}{c|}{46} & \multicolumn{1}{c|}{1.66} & 7.99e-11 & \multicolumn{1}{c|}{48} & \multicolumn{1}{c|}{1.29} & 4.17e-11 \\ \hline
				\texttt{poisson3Db} & 85623 & 2374949 & \multicolumn{1}{c|}{50} & \multicolumn{1}{c|}{\textbf{7.57}} & 6.94e-11 & \multicolumn{1}{c|}{51} & \multicolumn{1}{c|}{7.64} & 8.84e-11 & \multicolumn{1}{c|}{57} & \multicolumn{1}{c|}{8.76} & 6.66e-11 & \multicolumn{1}{c|}{59} & \multicolumn{1}{c|}{9.00} & 5.97e-11 \\ \hline
				\texttt{ifiss\_mat} & 96307 & 3599932 & \multicolumn{1}{c|}{33} & \multicolumn{1}{c|}{\textbf{2.27}} & 8.76e-11 & \multicolumn{1}{c|}{35} & \multicolumn{1}{c|}{2.32} & 3.38e-11 & \multicolumn{1}{c|}{42} & \multicolumn{1}{c|}{3.03} & 7.08e-11 & \multicolumn{1}{c|}{43} & \multicolumn{1}{c|}{2.73} & 9.70e-11 \\ \hline
				\texttt{hcircuit} & 105676 & 513072 & \multicolumn{1}{c|}{46} & \multicolumn{1}{c|}{0.80} & 9.69e-11 & \multicolumn{1}{c|}{46} & \multicolumn{1}{c|}{\textbf{0.38}} & 7.84e-11 & \multicolumn{1}{c|}{58} & \multicolumn{1}{c|}{1.44} & 4.99e-11 & \multicolumn{1}{c|}{58} & \multicolumn{1}{c|}{0.49} & 8.66e-11 \\ \hline
				\texttt{PR02R} & 161070 & 8185136 & \multicolumn{1}{c|}{61} & \multicolumn{1}{c|}{\textbf{25.13}} & 8.31e-11 & \multicolumn{1}{c|}{64} & \multicolumn{1}{c|}{26.39} & 5.15e-11 & \multicolumn{1}{c|}{100} & \multicolumn{1}{c|}{42.66} & 8.91e-11 & \multicolumn{1}{c|}{105} & \multicolumn{1}{c|}{46.23} & 4.92e-11 \\ \hline
				\texttt{cont-300} & 180895 & 988195 & \multicolumn{1}{c|}{105} & \multicolumn{1}{c|}{24.34} & 7.55e-11 & \multicolumn{1}{c|}{107} & \multicolumn{1}{c|}{\textbf{21.57}} & 9.34e-11 & \multicolumn{1}{c|}{147} & \multicolumn{1}{c|}{39.66} & 8.68e-11 & \multicolumn{1}{c|}{151} & \multicolumn{1}{c|}{32.55} & 9.49e-11 \\ \hline
				\texttt{thermomech\_dK} & 204316 & 2846228 & \multicolumn{1}{c|}{108} & \multicolumn{1}{c|}{14.06} & 9.26e-11 & \multicolumn{1}{c|}{110} & \multicolumn{1}{c|}{\textbf{8.59}} & 9.20e-11 & \multicolumn{1}{c|}{164} & \multicolumn{1}{c|}{27.30} & 8.81e-11 & \multicolumn{1}{c|}{167} & \multicolumn{1}{c|}{13.53} & 4.48e-11 \\ \hline
				\texttt{pwtk} & 217918 & 11524432 & \multicolumn{1}{c|}{190} & \multicolumn{1}{c|}{28.61} & 8.36e-11 & \multicolumn{1}{c|}{197} & \multicolumn{1}{c|}{\textbf{13.83}} & 8.64e-11 & \multicolumn{1}{c|}{283} & \multicolumn{1}{c|}{56.32} & 9.94e-11 & \multicolumn{1}{c|}{292} & \multicolumn{1}{c|}{21.83} & 7.55e-11 \\ \hline
				\texttt{Raj1} & 263743 & 1300261 & \multicolumn{1}{c|}{361} & \multicolumn{1}{c|}{103.21} & 9.79e-11 & \multicolumn{1}{c|}{398} & \multicolumn{1}{c|}{\textbf{80.74}} & 9.87e-11 & \multicolumn{1}{c|}{532} & \multicolumn{1}{c|}{239.82} & 9.91e-11 & \multicolumn{1}{c|}{567} & \multicolumn{1}{c|}{89.36} & 9.71e-11 \\ \hline
				\texttt{nxp1} & 414604 & 2655880 & \multicolumn{1}{c|}{105} & \multicolumn{1}{c|}{25.39} & 9.94e-11 & \multicolumn{1}{c|}{109} & \multicolumn{1}{c|}{\textbf{19.29}} & 7.62e-11 & \multicolumn{1}{c|}{125} & \multicolumn{1}{c|}{32.31} & 8.04e-11 & \multicolumn{1}{c|}{129} & \multicolumn{1}{c|}{24.67} & 8.35e-11 \\ \hline
			\end{tabular}
		\end{table}
	\end{landscape}

\section*{Declarations}
\subsection*{Funding} This work was supported by the National Natural Science Foundation of China (Grant numbers 12171403 and 11771364), and the Fujian Provincial Natural Science Foundation of China (No. 2025J01031).
\subsection*{Conflict of Interest} The authors have no competing interests to declare that are relevant to the content of this article.	
\subsection*{Data Availability} The data that support the findings of this study are available from the corresponding author upon reasonable request.
\subsection*{Author Contributions}
Both authors have contributed equally to the work.

{\small
\bibliographystyle{abbrv}
%\bibliography{pcmrh}

}	
	
\end{document}